\documentclass[10pt]{article} 
\usepackage[usenames,dvipsnames]{xcolor}

\RequirePackage[OT1]{fontenc}
\usepackage{underscore}
\usepackage[portrait,margin=3.5cm]{geometry}
\usepackage{mathrsfs}
\usepackage[colorlinks,citecolor=blue,urlcolor=blue,linkcolor=blue,linktocpage=true]{hyperref}
\usepackage{amssymb,amsthm,amsfonts,amsbsy,latexsym}
\usepackage[textsize=small]{todonotes}       
\usepackage{graphicx}
\usepackage{dsfont}
\usepackage[english]{babel}
\usepackage{subfigure}
\usepackage{appendix}

\usepackage{enumerate}

\usepackage{algorithmicx}
\usepackage{algorithm}

\usepackage{color}              
\usepackage{graphicx}
\usepackage[]{amsmath}
\usepackage{amsthm}
\usepackage{amssymb}
\usepackage{hyperref}
\usepackage{comment}
\usepackage{microtype}

\definecolor{MyDarkblue}{rgb}{0,0.08,0.50}
\definecolor{Brickred}{rgb}{0.65,0.08,0}

\newtheorem{theorem}{Theorem}[section]
\newtheorem{lemma}[theorem]{Lemma}

\newtheorem{remark}[theorem]{Remark}
\newtheorem{observation}[theorem]{Observation}

\newcommand{\E}[1]{\mathbb{E}\left[#1\right]}

\newcommand{\s}[2]{\sum_{#1}^{#2}}

\numberwithin{equation}{section}

\newcommand{\ErdosRenyi}{Erd\H{o}s-R\'enyi}

\renewcommand{\emptyset}{\varnothing}

\newcommand*{\be}{\begin{equation}}
\newcommand*{\ee}{\end{equation}}
\newcommand*{\ba}{\begin{aligned}}
\newcommand*{\ea}{\end{aligned}}
\newcommand*{\barr}{\begin{array}{c}}
\newcommand*{\earr}{\end{array}}

\def\namedlabel#1#2{\begingroup
    #2%
    \def\@currentlabel{#2}%
    \phantomsection\label{#1}\endgroup
}

\renewcommand{\P}[1]{\mathbb{P}\left(#1\right)}

\newcommand{\lr}[1]{\left(#1\right)}

\newcommand{\bigO}{\mathcal{O}}
\begin{document}

\title{A spectral method for community detection in moderately-sparse degree-corrected stochastic block models}
\author{Lennart Gulikers, Marc Lelarge, Laurent Massouli\'e}
\maketitle

\abstract{ \noindent \textit{We consider community detection in Degree-Corrected Stochastic Block Models (DC-SBM). We propose a spectral clustering algorithm based on a suitably normalized adjacency matrix. We show that this algorithm consistently recovers the block-membership of all but a vanishing fraction of nodes, in the regime where the lowest degree is of order log$(n)$ or higher. Recovery succeeds even for very heterogeneous degree-distributions. The used algorithm does not rely on parameters as input. In particular, it does not need to know the number of communities.
\\}

\section{Introduction}
Social and information networks are omnipresent in our daily lives and have been the interest of much recent research activity \cite{New10}. Studies have been focusing on local properties of network systems as well as their large-scale properties. Among those large-scale phenomena, community structure received a lot of attention. A wide variety of networks are found to have communities or blocks: groups of vertices with many links between themselves and substantially fewer to the rest of the network, or vice-versa. One of the fundamental problems in network inference considers the detection of such communities. See \cite{New04} and \cite{GiNe02} and references therein for an overview.

In the present manuscript we consider an instance of a certain probabilistic model that might be fit on the observed data. One of the best known such models is the stochastic block model (SBM) \cite{HoLaLe83}. In its simplest form, each of $n$ vertices belongs to precisely one of $K$ communities. Edges are independently drawn between different nodes with probabilities only depending on the block memberships of the involved vertices. This model is able to generate a diverse collection of random graphs, while it stays analytically tractable. 

In practice however, the SBM fails to accurately describe observed data: due to the stochastic inidentifiability of nodes in the same community, it does not allow for degree heterogeneity within blocks.
The DC-SBM was proposed in \cite{KaBrNe11} to overcome this issue. The DC-SBM allows, additional to a block-structure, the fitting of arbitrary degree sequences such that the expected degree of a vertex is independent of its community. The underlying paper deals with community detection in this model. 

Several methods for community detection can be found in the literature. They include, but are not limited to, modularity maximization \cite{NeGi04}, belief propagation \cite{DeAuKrFlMoCrZd11} and spectral clustering. For the latter, see for instance \cite{Lux07} and the section of related work in the underlying manuscript.  

Spectral algorithms employ eigenvectors of matrices representing network data to return non-local information of the network. The most commonly used matrices are the adjacency matrix and the (un)normalized graph laplacian \cite{Lux07}. In \cite{ZhNaNe14} the authors study the spectra of the adjacency matrix for networks possessing arbitrary degree distributions while simultaneously exhibiting a community-structure. 
They demonstrate that those spectra consist in general of two components: a  part containing the bulk of eigenvalues and a separated part with outliers whose number is in general equal to the number of blocks present. 

The contribution of our paper is as follows: We demonstrate with a clean analysis that community detection in a moderately sparse DC-SBM is feasible under rather general conditions on the degree-sequence. 

 More specifically, we consider the matrix $\widehat{H}$ with entry $(u,v)$ given by $\widehat{H}_{uv} = \frac{1}{\widehat{D}_u \widehat{D}_v} A_{uv}$ if $A_{uv} = 1$ and $\widehat{H}_{uv} = 0$ otherwise (here $A$ is the adjacency matrix of the graph and $\widehat{D}_u$ is the observed degree of vertex $u$), which we shall call  the \emph{normalized adjacency matrix}. 

 We show that this matrix concentrates around a deterministic matrix $P$ of rank $L \leq K$, when the minimum expected degree is as small as log($n$). To establish this concentration-result, we use Lemma \ref{lm::alignment_eigenvectors} below, which could be of independent interest, as a simple alternative to the commonly used Davis-Kahan theorem. 

 Due to the underlying community structure, the matrix that has the first $L$ eigenvectors of $P$ as its columns has the nice property that it has only $K$ different rows. Hence, due to this fact and the concentration of $\widehat{H}$ around $P$,  the rows of the corresponding eigenvector matrix of $\widehat{H}$ considered as points in an $L$-dimensional euclidean space, must cluster around $K$ centres. This property indicates that $\widehat{H}$ is the right matrix to analyse when dealing with the DC-SBM. Indeed, associating each vertex  with its corresponding row, we show in this paper that we retrieve the correct community of all but a vanishing fraction of nodes. 

 Further, we point out a natural connection between $\widehat{H}$ and a random walk on the observed graph. 


The organization of this article is as follows: First we formally introduce the DC-SBM together with necessary notations. Next we state our main result for community detection in this model, followed by a discussion in Section 4: a discussion of relevant literature, \textcolor{black}{performance on real data}, the conditions in the main theorem and a connection between $\widehat{H}$ and random walks.  Section 5 outlines the approach we take to prove the main theorem, which is accompanied by a statement of all auxiliary lemmas. Section 6 contains algebraic preliminaries. All proofs are deferred to  section 7. In the last section we give a suggestion for future research.

\section{Model and notations}
The Degree-Corrected Stochastic Block Model denoted by $\mathcal{G}(B,K,\{ \sigma_u \}_{u=1}^n,\{D_u \}_{u=1}^n)$ is a generalization of the Erd\"os-R\'enyi classical model of random graphs. We introduce a random graph on $V := \{1, \ldots, n\}$. We partition the set of vertices into $K$ communities of $\alpha_k n$ members each: each vertex $u$ is given a label $\sigma_u \in S:= \{1, \ldots, K\}$.   A weight $D_u$ is given to each vertex $u$ to encode its expected degree. Without loss of generality we assume that $D_1 \leq D_2 \leq \cdots \leq D_n$. All weights and labels will depend on $n$, but this is suppressed in the notation here. For each pair $(u,v)$, we include the edge $(u,v)$ with probability
\be \P{u \sim v} = \left\{ 
  \begin{array}{l l}
    \frac{D_u D_v}{n \overline{D}} B_{\sigma_u \sigma_v} & \quad \text{ if } u \neq v\\
    0 & \quad  \text{ if } u =v, \\
  \end{array}  \right. \label{eq::Prob} \ee
where $B \in (\mathbb{R}^{+})^{K \times K}$ is a symmetric matrix, independent of $n$ and $\overline{D} = 1/n \s{l = 1}{n} D_l$, the average weight.  $B$ may be chosen \emph{completely} independent of the weights $\{D_u\}_{u=1}^n$: all information about the community-structure is then captured by $B$ alone. 

We make some further assumptions on the parameters of the model:
For \eqref{eq::Prob} to define a probability, we assume 
\be \frac{D_u D_v}{n \overline{D}} B_{\sigma_u \sigma_v} \leq 1, \label{eq::require_prob} \ee
for all $u,v$. 

 The vector $\boldsymbol \alpha = (\alpha_1, \ldots, \alpha_K)$ is assumed to be constant. Hence, the clusters are well balanced, as the size of each community grows linearly with $n$. Further, the average weight in a cluster, 
\[ \overline{D}_i = \frac{1}{\alpha_i n} \s{u=1}{n} D_u \mathds{1}_{\sigma_u = i}, \]
is assumed to be asymptotically a fraction of the average weight $\overline{D}$. That is, we assume that there exists non-zero constants $d_1, \ldots, d_K$, such that,
\be \lim_{n \to \infty} \frac{\overline{D}_i }{\overline{D}} = d_i, \label{eq::d_i_convergence} \ee
for all $i$. Under this assumption, the following limit exists for all $i$,
\be \overline{M}_i = \lim_{n \to \infty} \frac{M_i}{\s{l=1}{n} D_l} = \s{k=1}{K} B_{ik} \alpha_k d_k \label{eq::M_i_convergence},\ee
\textcolor{black}{where $M_i = \s{l}{} D_l B_{i \sigma_l}$}.

 We shall see in Section 4.4 that we need the following condition for the communities to be identifiable: we assume that for all $i,l$ there exists $j$ such that
\be  \frac{B_{ij}}{\overline{M}_i} \neq \frac{B_{lj}}{\overline{M}_l}.  \label{eq::dist} \ee

 In the analysis that follows, we will consider the random graph in a moderately-sparse regime, that is we assume:
 either \be \lim_{n \to \infty} \frac{D_1}{\text{log}(n)} = \infty, \label{eq::growth_sup_log} \ee
  or, for some constant $c < 1/2$, \be D_1 \geq C_{B,\textbf{M}} \cdot \text{log}(n)  \mbox{  and   }  \lim_{n \to \infty}  \frac{D_n^2 }{ n^c} \to 0, \label{eq::growth_log} \ee
where  $C_{B,\textbf{M}}$ is some constant depending on $B$, $\textbf{M} = ( \overline{M}_1, \ldots \overline{M}_K  )$ and the convergence rate in \eqref{eq::M_i_convergence}.
 Further, we assume the following condition on the weights:
\be \frac{D_1^2}{\overline{D}} =  \Omega( \text{log}(n)). \label{eq::degree_hetero} \ee 
Note that under those assumptions, $D_u$ represents the expected degree of vertex $u$ upto a multiplicative factor that depends only on the community $\sigma_u$. Indeed, if $\widehat{D}_u$ denotes the observed degree of vertex $u$, then
\be \ba \E{\widehat{D}_u} &= \frac{D_u}{n \overline{D}} \s{l \neq u}{} D_l B_{\sigma_u \sigma_l} \\
&= \frac{D_u}{n \overline{D}} \lr{ M_{\sigma_u} - D_u B_{\sigma_u \sigma_u} } \\
&= D_u \overline{M}_{\sigma_u} (1 + \epsilon_n)  , \label{eq::expD} \ea \ee
where  $ \epsilon_n \leq \frac{2}{\overline{M}_{\sigma_u}  D_v} = o_n(1)$.

As an example, we let $\{\sigma_u \}_{u=1}^n$ be any sequence such that $n/2$ of its elements are $1$ and the other $n/2$ elements are $2$. Then, there are two equally-sized communities: $K=2$ and $\alpha_1 = \alpha_2 = 1/2$. Let $\{ D_u\}_{u=1}^n$ be any non-decreasing sequence  with $D_1 > 0$. Put
\[ B =  \left( \begin{array}{ccc} a & b  \\
b & a  \end{array} \right), \]
for some constants $a$ and $b$. Then
\[  \P{u \sim v} =  \frac{D_u D_v}{n \overline{D}}   \left\{ 
  \begin{array}{l l}
    a & \quad \text{ if }  \sigma_u = \sigma_v,\\
    b  & \quad  \text{ otherwise }. \\
  \end{array} \right.\]
This is exactly the extended-planted partition model (EPPM) considered in \cite{ChChTs12}.
  
\section{Main Results}
Our aim is to retrieve the underlying community structure from a single observation of the random graph. We do this by analysing the spectral properties of $\widehat{H} \in \mathbb{R}^{n \times n}$ defined for $u,v \in V$ by
 \begin{equation}
 \widehat{H}_{uv} = \left\{ 
  \begin{array}{l l}
    \frac{1}{\widehat{D}_u \widehat{D}_v} A_{uv} & \quad \text{ if } A_{uv} = 1,\\
    0 & \quad  \text{ otherwise }, \\
  \end{array} \right.
  \label{eq::widehat_H}
  \end{equation}
  where $A$ is the adjacency matrix of the observed graph.
We shall demonstrate that this matrix is close (in a sense to be specified below) to the matrix $P$ defined for $(u,v)$ as
\be P_{uv} = \frac{1}{n \overline{D}} \frac{B_{\sigma_u \sigma_v}}{\overline{M}_{\sigma_u} \overline{M}_{\sigma_v}}. \label{eq::Matrix_overlineP} \ee
Denote the rank of $P$ by $L$. Due to the community structure, $L \leq K$ (see below for details). 

In the regime where \eqref{eq::growth_sup_log} holds, let $f$ be any function tending to zero, such that \[f(n) \gg \frac{1}{\widehat{D}_1} + \frac{1}{\sqrt{\text{log}(n)}} + \sqrt{\frac{\text{log}(n)}{\widehat{D}_1}}.\] For the regime where \eqref{eq::growth_log} holds, let $f$ be tending to zero in such a way that \[f(n) \gg \frac{1}{\widehat{D}_1} + \frac{1}{\sqrt{\text{log}(n)}} + \frac{1}{\text{log}^{1/3}(n)}.\] Further, let $\tau(n) = 1 / f(n)^{1/3}$.  

 Algorithm $1$ uses $\widehat{H}$ to reconstruct the communities. \newpage
\begin{algorithm}[h]                      
\caption{}          
\label{alg1}                           
\begin{algorithmic}                    
   
    \State 
\begin{enumerate}
\item  Calculate the average degree in the graph, call it $\widehat{D}_{\text{average}}$. Let $\widehat{L}$ be the number of eigenvalues of $\widehat{H}$ that are in absolute value larger than $f(n) / \widehat{D}_{\text{average}}$.
\item Compute the first  $\widehat{L}$ orthonormal eigenvectors of $\widehat{H}$ ordered according to their absolute eigenvalues. Denote these eigenvectors and their corresponding eigenvalues by $\widehat{x}_1, \ldots, \widehat{x}_{\widehat{L}}$ and $\widehat{\lambda}_1, \ldots, \widehat{\lambda}_{\widehat{L}}$ respectively.
\item Associate to each node $u \in V$ the vector
\be \widehat{z}_u = ( \widehat{x}_1(u) , \ldots , \widehat{x}_{\widehat{L}}(u) ). \label{def::lines} \ee Cluster the vectors $(\widehat{z}_u)_{u=1}^n$ as follows: Pick $\tau(n)$ pairs of vertices, label them $(u(1),u'(1)), \ldots, (u(\tau(n)),u'(\tau(n)))$. Calculate $\delta(t) = \sqrt{n}||\widehat{z}_{u(t)} - \widehat{z}_{u'(t)} ||$, and $\epsilon = \min_{t: \delta(t) > f^{2/3}(n)} \delta(t)$.  Find a vertex $m$ so that $\{u' : \sqrt{n} ||\widehat{z}_{m} - \widehat{z}_{u'} || \leq \epsilon / 8 \}$ has cardinality larger than $f^{1/3}(n) \  n $. Form a community consisting of all nodes in $\{u' : \sqrt{n} ||\widehat{z}_{m} - \widehat{z}_{u'} || \leq \epsilon / 4 \}$. Remove those nodes and iterate this procedure.
\end{enumerate}    
\end{algorithmic}
\end{algorithm}

 We have:
\begin{theorem}
Consider a DC-SBM $\mathcal{G}(B, \boldsymbol \alpha ,\{ \sigma_u \}_{u=1}^n,\{ D_u \}_{u=1}^n)$. Assume assumptions \eqref{eq::require_prob}, \eqref{eq::d_i_convergence}, \eqref{eq::dist},  \eqref{eq::degree_hetero} and either \eqref{eq::growth_sup_log} or \eqref{eq::growth_log} to hold. Then, \emph{Algorithm 1} retrieves the community of all but a vanishing fraction of nodes. 
\label{thm::main}
\end{theorem}

 The first step estimates $L$. Indeed, by definition there are only $L$ non-zero eigenvalues of $P$. Those are all of order $1 / \overline{D}$ and the corresponding first eigenvalues of $\widehat{H}$ are of the same order. The remaining eigenvalues  of $\widehat{H}$ are negligible with respect to $f(n) / \overline{D}$.

Under the assumptions in Theorem \ref{thm::main}, all but a negligible number of rows of the matrix having the first $L$ eigenvectors of $\widehat{H}$ as it columns, cluster for large $n$ to within negligible distance of block-specific representatives that are separated by some non-vanishing gap (call the corresponding vertices typical). This is exploited in the third step. There, with high probability, all picked vertices are typical. Thus, for a pair $t$, $\delta(t)$ vanishes in front of $f^{2/3}(n)$ if the vertices in the pair belong to the same community. Hence, by calculating the distance between the other vertices, we obtain $\epsilon$ as an estimator for the gap mentioned above. At most $f(n)^{2/3} \ n$ vertices are not typical. Hence, the chosen ball around $m$ with radius $\epsilon / 8$ contains a negligible number of non-typical vertices, the remaining vertices should necessarily be in the same community. By enlarging the radius of the ball around $m$, we include all vertices of a single community. See the proof of Theorem \ref{thm::main} below for more details. 

\begin{remark}
Note that the only input to the algorithm is the regime (i.e., either $D_1(n) = \Theta(\emph{log}(n))$ or $D_1(n) \gg \emph{log}(n)$). This information is used to pick the right form of the function $f$. Alternatively, we could adapt the algorithm so that it requires $L = \text{Rank}(B)$ and $\alpha_{min}$ instead of the regime: Step $1$ can then be skipped, in Step $2$ we replace $\widehat{L}$ by $L$ and in Step $3$ we chose a vertex $m$ that contains in its $\epsilon/8$ neighbourhood at least $\alpha_{min} n /2$ vertices.
\end{remark}
\section{Discussion}
Before we prove the main theorem, we make some observations and remarks.

\subsection{Adjacency matrix}
In \cite{LeRi13} and  \cite{McS01}, the authors use the adjacency matrix $A$ of a graph to recover the underlying community-structure. They consider the matrix having the first $K$ eigenvectors of $A$ as its columns and show that, under appropriate conditions, its rows cluster now in $K$ different \emph{directions}. 
However, results in \cite{ChLuVu03} and  \cite{MiPa02} suggest that the algorithms in  \cite{LeRi13} and  \cite{McS01} fail when the expected degree sequence is too irregular. Intuitively, if the prescribed degree sequence follows a power-law, then so does the spectrum of the adjacency matrix. Further, as we shall demonstrate below, the first $K$ eigenvectors correspond only to the $K$ top-degree nodes, and should therefore not be expected to capture more global features of a graph, such as its underlying block-structure. The following theorem makes this observation more rigorous:
\begin{theorem}
Consider a DC-SBM $\mathcal{G}(B,K,\{ \sigma_u \}_{u=1}^n,\{ D_u \}_{u=1}^n)$ such that
\be D_u = \left\{ 
  \begin{array}{l l}
    D_1 & \quad \text{ if } 1 \leq u < n - k \\
    D_1 n^{\gamma} (u + 1 - (n-k)) & \quad  \text{ if } u \geq n - k, \\
  \end{array} \right. \ee
  where $k = n^{\beta}$ and the constants $\beta$ and $\gamma$ obey:
 \be D^2_1(n) n^{2 \gamma + 4 \beta - 1} \to 0 \label{eq::C_beta_gamma_1} \ee
 and
 \be \gamma  > 4 \beta. \label{eq::C_beta_gamma_2} \ee
Further, assume that
\be \sigma_u = \left\{ 
  \begin{array}{l l}
    2 & \quad \text{ if } u \leq \frac{n}{2}\\
    1 & \quad  \text{ if } u > \frac{n}{2}. \\
  \end{array} \right. \ee
Under these conditions, the first $k$ eigenvectors become for large $n$ indistinguishable from the eigenvectors of a graph that is the disjoint-union of $k$ stars having degrees $D_n + o(1), \ldots, D_{n-k} + o(1).$
\label{thm::power_law}
\end{theorem}

 For instance, $D_1(n) = n^{1/20}$, $\beta = 1/20$ and $\gamma = 1/5$, meets the assumptions in Theorem \ref{thm::power_law}. Further, it verifies the conditions in the main theorem (Theorem \ref{thm::main}): Algorithm $1$ will successfully return the community membership of all but a vanishing fraction of nodes. 

We remark that the above theorem is inspired by the main result in \cite{MiPa02}. There, random graphs without community structure are considered and the power-law behaviour of the corresponding spectrum is obtained. To say something about the eigenvectors, we additionally introduce a gap between the top $k$ degreed-nodes and the remaining $n-k$ nodes. This allows us to use Lemma \ref{lm::alignment_eigenvectors}, see the proof of \ref{thm::power_law} below.  

\subsection{SCORE}
\textcolor{black}{Interestingly, the first eigenvectors of $A$ do contain information about the underlying community structure, but in a hidden way. Indeed, the SCORE method proposed in \cite{Jin15} shows that, under some conditions, using the coordinate-wise ratios of the leading eigenvectors leads to consistent clustering.}

\textcolor{black}{Note that we obtain the same random graph model as in \cite{Jin15} by putting $\theta(u):= D_u / \sqrt{n \overline{D} \alpha }$ and $P(i,j) = \alpha B_{\sigma_u \sigma_v},$ where $\alpha ^{-1}= \max_{i,j} B_{ij}.$ We further note that the conditions are more stringent: $(2.7)$ demands that $P$ (or $B$) is non-singular which is unnecessary here, see Remark \ref{rm::violating_ident} below. 
}

\subsection{Laplacian}
As we just pointed out, the adjacency matrix does not capture accurately global properties of a graph. The normalized Laplacian is a more suitable candidate. It is defined by $L = I - D^{-1/2} A D^{-1/2}$, where $I$ is the identity matrix, $A$ is the adjacency graph and $D$ the diagonal matrix with the row sums of $A$ on its diagonal (i.e., the degrees). Object of study in \cite{ChLuVu03} is the Laplacian spectra of random graphs with a given degree sequence $(d_1, \ldots, d_n)$ where edges are independently present between each pair of vertices $(u,v)$ with probability $\frac{d_u d_v}{\s{l=1}{n} d_l}$. In the regime $d_1^2 \gg \overline{D}$, with $\overline{D} = 1/n \s{l=1}{n} d_l$, the eigenvalues satisfy the semicircle law with respect to the circle of radius $2 / \sqrt{\overline{D}}$ centred at $1$. 

Denote the eigenvalues of the normalized laplacian by $0 = \lambda_1 \leq \lambda_2 \leq \ldots \leq \lambda_n \leq 2$. It is a well-known fact that all eigenvalues are located in the interval $[0,2]$ and that the algebraic multiplicity of $0$ equals the number of components in the graph. The authors of \cite{ChLuVu03} further study the spectral gap $\lambda = \min \{ \lambda_2, 2 - \lambda_n \}$, which reflects global properties of the random graph. 
According to  \cite{ChLuVu03}, when $d_1 \gg \text{log}^2(n)$, 
\[ \lambda \geq 1 - \frac{1+o(1)}{4/ \sqrt{\overline{w}}} - \frac{\text{log}^2(n)}{d_1}, \]
thus in this dense regime, all non-zero eigenvalues are close to $1$ and thus the spectrum of the Laplacian contains no outliers, in contrast with the adjacency matrix. This bound is improved in \cite{ChRa11}, to
\[ \lambda \geq 1 - 2 \sqrt{\frac{6 \text{log}(2n)}{d_1}}, \]
for $d_1 \gg \text{log}(n)$. 

The stochastic block model is a special case of the latent space model \cite{HoRaHa01}. In this model a vector $v_u$ is associated to each node $u$ and an edge between $u$ and $v$ is present with probability depending only on $z_u$ and $z_v$. If $A$ is the adjacency matrix of the graph, $D$ the diagonal matrix containing the degrees and $L = D^{-1/2} A D^{-1/2}$, then the population version of these matrices are defined as
$ \mathcal{A} = \E{A | z_1, \ldots, z_n}, $
$ \mathcal{D} = \text{diag} \lr{ \s{v=1}{n} \mathcal{A}_{1v}, \ldots, \s{v=1}{n} \mathcal{A}_{nv} },  $
and, $ \mathcal{L} = \mathcal{D}^{-1/2} \mathcal{A} \mathcal{D}^{-1/2}. $
In \cite{RoChYu11} convergence of the empirical eigenvectors of $L$ to the population eigenvectors of $\mathcal{L}$ is shown. This follows from their novel result establishing the convergence of $L^2$ to $\mathcal{L}^2$ in Frobenius norm. This forms the basis of an algorithm that uses the first $k$ eigenvectors (according to the eigenvalues order decreasingly with respect to their absolute value). To recover the hidden communities in the SBM (thus, without degree-corrections). The algorithm is shown to succeed if those first $k$ eigenvalues are sufficiently separated from the rest of the eigenvalues and if the minimum expected degree exceeds $\frac{\sqrt{2}n}{\sqrt{\text{log} n}},$ which is more restrictive than the lower bound of $\text{log} n$. 

In \cite{DaHoMc04} the matrix $\E{D}^{-1/2} A \E{D}^{-1/2}$ (reminiscent of the normalized Laplacian) is used to retrieve the underlying community structure in the DC-SBM. Note that this method requires the expected degrees to be known. It succeeds if the minimum degree is of order $\text{log}^6 n$. 

To deal with low-degree nodes, the authors in \cite{ChChTs12} use the degree-corrected random walk laplacian: $I - (D + \tau I)^{-1}A,$
where $\tau > 0$ is a constant, to find clusters in the extended planted partition mode (EPPM) where the expected minimum degree is $\Omega( \text{log} n )$. In the EPPM, $B$ is a matrix where an element equals $p$ if it is on the diagonal, and $q$ otherwise; it is thus a special case of the DC-SBM. The algorithm based on the random walk laplacian requires $\tau$ as input and the optimal value of $\tau$ depends in a complex way on the degree-distribution of the graph.  The main theorem in \cite{ChChTs12} comes with lengthy conditions that are not easy to compare with other results. This theorem restricted to the setting where all $d_u$'s equal $d$,  assumes $q$ to be a constant, which is more restrictive than our assumptions. It is unclear whether the results for the EPPM can be neatly generalized using the same operator to the DC-SBM, given the complexity of the present conditions.

Although the Laplacian captures global properties of a graph much better than the adjacency matrix, its spectrum is still influenced by the underlying degree-structure. Indeed, consider a DC-SBM with $3000$ vertices divided in $K=3$ equally-sized communities, with 
\[
B = \lr{\begin{matrix} 
1 & 2 & 3 \\
2 & 0 & 2 \\
3 & 2 & 5 \\
\end{matrix}},
\]
degree-sequence
\be D_u = \left\{ 
  \begin{array}{l l}
    u^{1/3} & \quad \text{ if } u=1, \ldots,1000 \\
    (u-1000)^{1/3} & \quad \text{ if } u=1001, \ldots,2000 \\
    (u-2000)^{1/3} & \quad \text{ if } u=2001, \ldots,3000, \\
  \end{array} \right. \ee
  and community-membership
\be \sigma_u = \left\{ 
  \begin{array}{l l}
    1 & \quad \text{ if } u=1, \ldots,1000 \\
    2 & \quad \text{ if } u=1001, \ldots,2000 \\
    3 & \quad \text{ if } u=2001, \ldots,3000. \\
  \end{array} \right. \ee
In Figure \ref{fig::L}, we have plot the eigenvectors corresponding to the first and second largest absolute eigenvalue of $I - \E{D}^{-1/2} \E{A} \E{D}^{-1/2}$, where $A$ is the adjacency matrix and $D$ is the diagonal matrix containing the row sums of $A$. The Laplacian concentrates around $I - \E{D}^{-1/2} \E{A} \E{D}^{-1/2}$ if the minimum degree is large enough (see Section $8$). The community structure is clearly perturbed by the degree-sequence. In general, an additional step is needed to determine the community-membership of all nodes when using the Laplacian.  

Compare this figure to Figure \ref{fig::H}, containing the first two eigenvectors of \\ $\E{D}^{-1} \E{A} \E{D}^{-1}$. The vertices are seen to be clearly divided into three communities. 
\begin{figure}[!ht]
\begin{minipage}[t]{0.45\linewidth}
\centering
\includegraphics[scale=0.4]{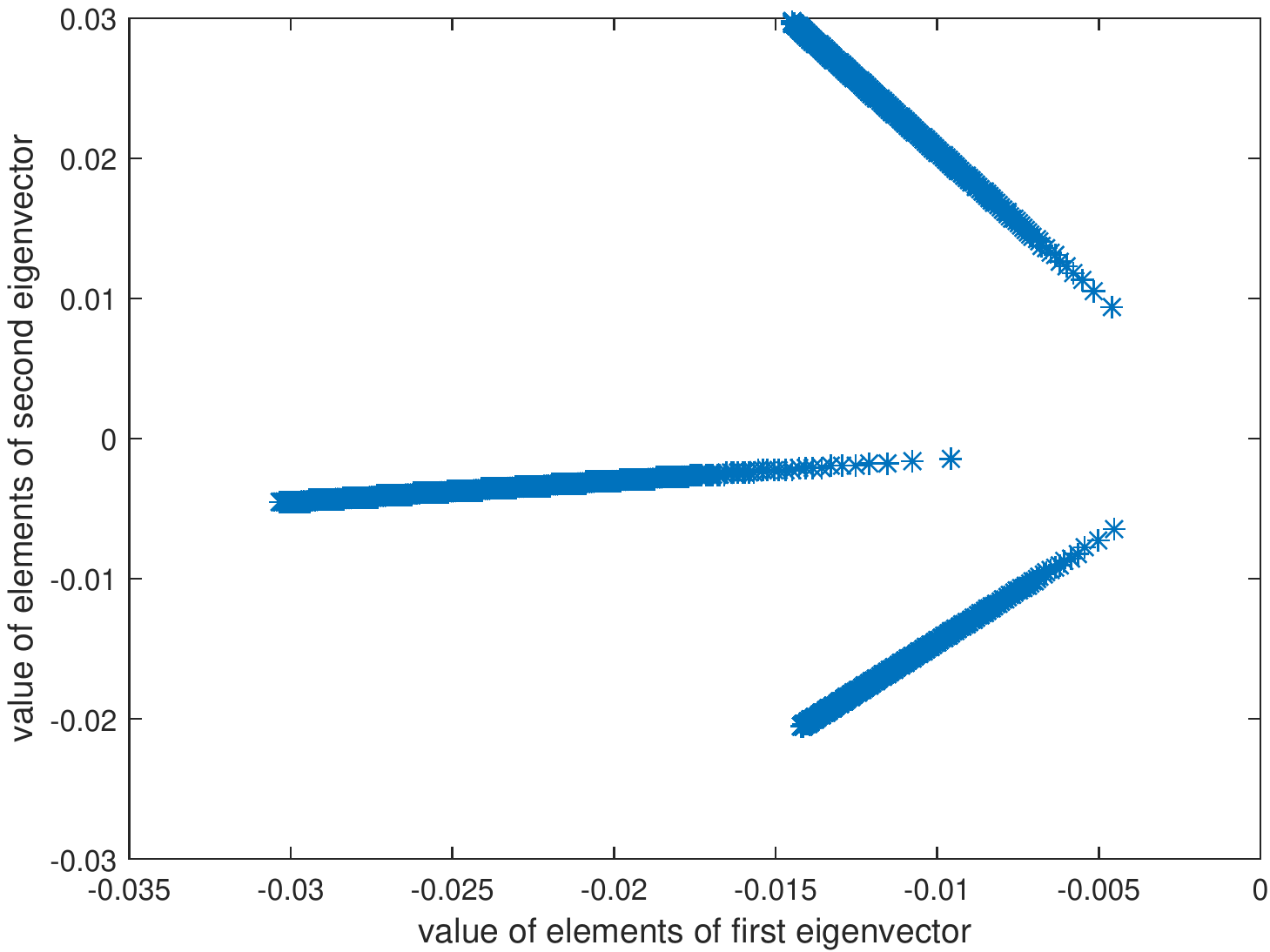}
\caption{Plot of the eigenvectors corresponding to the first and second largest absolute eigenvalue of $I - \E{D}^{-1/2} \E{A} \E{D}^{-1/2}$, where $A$ is the adjacency matrix of a random graph drawn according to the DC-SBM defined at the end of Section 4.2, and $D$ is the diagonal matrix containing the row sums of $A$. 
 For those eigenvectors, say $(x_1, \ldots, x_n)'$ and $(y_1, \ldots, y_n)'$, we draw a dot $(x_u,y_u)$  for each element $u$.}
\label{fig::L}
\end{minipage}
\hspace{0.5cm}
\begin{minipage}[t]{0.45\linewidth}
\centering
\includegraphics[scale=0.4]{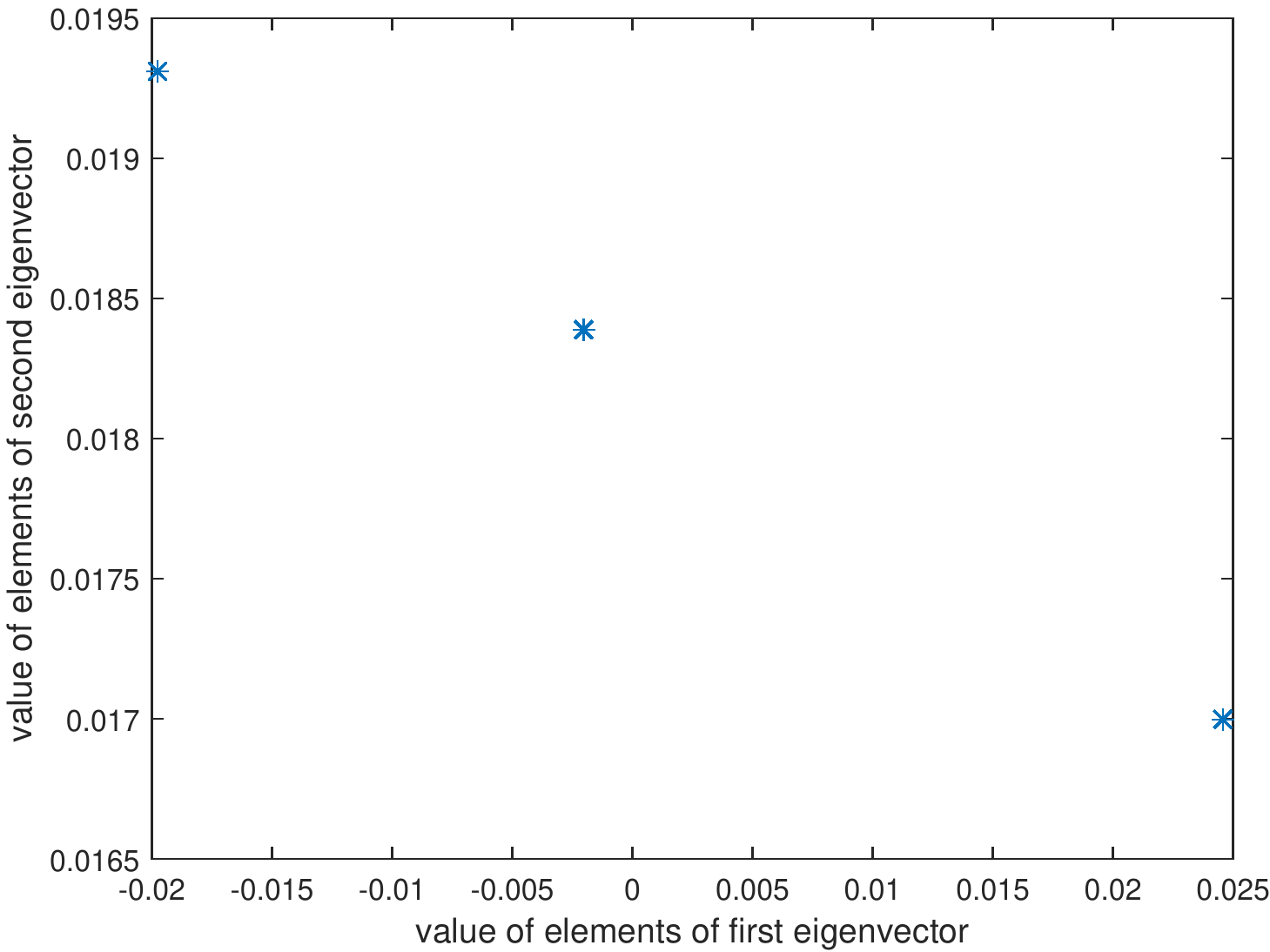}
\caption{Plot of the eigenvectors corresponding to the first and second largest absolute eigenvalue of $\E{D}^{-1} \E{A} \E{D}^{-1}$, where $A$ is the adjacency matrix of a random graph drawn according to the DC-SBM defined at the end of Section 4.2, and $D$ is the diagonal matrix containing the row sums of $A$. 
  For those eigenvectors, say $(x_1, \ldots, x_n)'$ and $(y_1, \ldots, y_n)'$, we draw a dot $(x_u,y_u)$  for each element $u$. Note that many elements are represented by the same dot, clearly reflecting the community structure.}
\label{fig::H}
\end{minipage}
\end{figure}

\textcolor{black}{
Now consider another two-community DC-SBM on $n$ vertices with \[
B = \lr{\begin{matrix} 
1 & 1 \\
1 & 1 \\
\end{matrix}},
\]
degree-sequence
\be D_u = \left\{ 
  \begin{array}{l l}
    \log^2(n) & \quad \text{ if } u \leq n/2 \\
    100 \log^2(n) & \quad \text{ if } u > n/2.
  \end{array} \right. \ee
  and community-membership
\be \sigma_u = \left\{ 
  \begin{array}{l l}
    1 & \quad \text{ if } u \leq n/2 \\
    2 & \quad \text{ if } u > n/2. \\
  \end{array} \right. \ee}
  
  \textcolor{black}{
Then, according to Lemma \ref{lm::alignment_eigenvectors}, the eigenvectors of $H$ become eventually indistinguishable from the eigenvectors of the $n \times n$ matrix with zero-diagonal and all other elements equal to $\frac{1}{n \overline{D}}$. Clearly the communities can not be recovered from the latter matrix.}

\textcolor{black}{
The off-diagonal elements of $\E{D}^{-1/2} \E{A} \E{D}^{-1/2}$ are given by $\frac{1}{n \overline{D}} \sqrt{D_u} \sqrt{D_v} = \frac{2}{101n} Z_{\sigma_u \sigma_v}$, with
$Z = \lr{\begin{matrix} 
1 & 10 \\
10 & 100 \\
\end{matrix}}.
$
Now, $Z$ has eigenvector $\lr{\begin{matrix} 
1 \\
10 \\
\end{matrix}},
$ corresponding to eigenvalue $101$. The other eigenvalue is zero. So that the minimal gap between different eigenvalues of $\E{D}^{-1/2} \E{A} \E{D}^{-1/2}$ is $2 - \bigO(1/n)$. According to \cite{ChRa11}, $\rho \lr{ D^{-1/2} A D^{-1/2} - \E{D}^{-1/2} \E{A} \E{D}^{-1/2} } = o(1)$ w.h.p., where $\rho(X)$ denotes the spectral radius of a matrix $X$. Consequently, Lemma \ref{lm::alignment_eigenvectors} entails that for large $n$, clustering according to the  eigenvector of $D^{-1/2} A D^{-1/2}$, corresponding to its largest eigenvalue, reveals the community-membership of all but a vanishing fraction of nodes.}

\textcolor{black}{
Those two examples hint that whether the Laplacian $L$ or the degree-normalized adjacency matrix $H$ should be used depends on the correlation between the degrees and the communities, and the 'signal-strength' of $B$. The first example shows that if the degrees are uncorrelated, $L$ seems to add some extra noise, whereas $H$ 'filters' the degrees and reflects immediately the underlying communities. In the second example, $B$ gives no information about the communities, but the vertices can be clustered according to their degrees. $H$ ignores this degree-structure and thus fails to detect the communities. $L$ on its turn still reflects the degree-sequence and therefore the communities.}

\subsection{Regularized spectral clustering}
\label{ss::regularized}
\textcolor{black}{The paper \cite{QiRo13} deals with the shortcomings of the Laplacian by inflating the degrees: Given a number $\tau>0$, the regularized graph Laplacian \cite{QiRo13,ChChTs12} is defined as 
\be L_{\tau} = D_{\tau}^{-1/2} A D_{\tau}^{-1/2}, \label{eq::reg_laplacian} \ee
where $D_{\tau} = D + \tau I$.}

\textcolor{black}{The regularized spectral clustering algorithm in \cite{QiRo13} starts with computing the matrix $X= [X_1, X_2, \cdots, X_K]$, where $X_1, X_2, \ldots, X_K$ are the eigenvectors corresponding to the $K$ largest eigenvalues. A matrix $X^{*}$ is then formed by projecting each row of $X$ on the unit sphere. Considering each row of $X^{*}$ as a point in $\mathbb{R}^K$, and applying $k$-means with $K$ centres on these points gives an almost-exact clustering provided some conditions on $\delta + \tau$ ($\delta$ is the smallest expected degree) and the smallest strictly positive eigenvalue of $L_{\tau}$ hold. In particular, condition $(a)$ in Theorem $4.2$ demands that $\delta + \tau \gg$ log$(n)$. Since simulation results suggest that $\tau$ should be taken as the average degree, it is unclear if this method outperforms the algorithm proposed in the underlying paper.} 

\textcolor{black}{We note that \cite{QiRo13} is the first work that relates the leverage scores (the euclidean norm of the rows of $X$) to the quality of the outputted clustering.} 

\subsection{When does the degree-normalized adjacency matrix fail?}
\textcolor{black}{Consider a DC-SBM with $K \geq 2$ communities, such that for two different communities $i \neq j$, for all $l$,
$\frac{B_{il}}{\overline{M}_i} = \frac{B_{jl}}{\overline{M}_j}.$
Then, it can be verified that, for large $n$, in a dense enough regime, the eigenvectors of $H$ corresponding to non-zero eigenvalues, do not distinguish between communities $i$ and $j$.}

\textcolor{black}{Further, the method breaks down in a too sparse regime. For instance, two low-degreed vertices connected by an edge cause the top eigenvectors to concentrate around them. We observed this when applying $\widehat{H}$ on the sparse Political Blogs network \cite{AdGl05}, see Section \ref{ss::real_data}.}

\subsection{Degree-normalized adjacency matrix}
The same matrix $H$ is used in \cite{CoLa09} to perform community detection on the DC-SBM in the sparse regime (the minimum degree is bounded from below by a constant). The main restriction in their setting is that the minimum degree must be of the same order as the average degree, more precisely there exists $\epsilon > 0$ such that $D_i > \epsilon \overline{D}$ for all $i$. Hence too much irregularity in the degree sequence is not captured. In this sense our work complements their results. 

Spectral clustering is performed in \cite{CoLa09} on a minor of $\widehat{H}$ where the rows and columns of vertices with a degree smaller than $D_{\text{average}} / \text{log}(n)$ (where $D_{\text{average}}$ is the observed average degree in the graph) are put to zero, which is not the same as leaving out completely the nodes with a too low degree. Due to the assumption that all expected degrees are of the same order, most observed degree will exceed the lower bound  $D_{\text{average}} / \text{log}(n)$. 

There are alternative ways to deal with low degree nodes, see for instance section 8 on future research. 

\subsection{Performance on real networks: Karate Club, Dolphins and Political Blogs}\label{ss::real_data}
\textcolor{black}{We have tested our method on 3 real networks, namely, Zachary's karate club \cite{Zac77}, the dolphin social network \cite{LuScBoHaSlDa03} and the political blogs dataset \cite{AdGl05}.  The error rate for Zachary's karate club is $2/34$ and for the dolphin social network $0/62$.}
 
\textcolor{black}{The error rate for the political blogs dataset is $230/1221$ when thresholding the Frobenius eigenvector. We restricted to the giant component of $1221$ nodes, as is common in most other works (the original data contained $1490$ blogs). Our clustering is worse than obtained by SCORE (where the error rate is $58/1221$), but similar to the non-backtracking matrix  (where around $15$ percent of the nodes are misclassified \cite{KrMoMoNeSlZdZh13}).}

\textcolor{black}{We observed that the leading eigenvectors are concentrated on a few nodes, due to the presence of certain problematic structures (such as two low-degreed vertices connected by an edge). However, the value of the Frobenius eigenvector on the remaining vertices is still correlated with their community-membership as can be observed in Figures \ref{fig::concentration} and \ref{fig::ranking}.}

\textcolor{black}{Figure \ref{fig::concentration} is a histogram of the Frobenius eigenvector restricted to the roughly $600$ nodes that have corresponding value in the interval $[0,10^{-9}]$. The nodes seem to concentrate around two centres according to their community. However, this phenomenon is only weakly visible (note that our theory does not apply for \emph{sparse} graphs).}

\textcolor{black}{In Figure \ref{fig::ranking} we have sorted the $1221$ indices of the Frobenius eigenvector according to an increasing corresponding value: the community structure becomes then clear.}

\textcolor{black}{We further observed that thresholding the eight-est eigenvector leads to only $160$ misclassified vertices. Interestingly, if we inflate the degrees by replacing  $H = \frac{A_{uv}}{\widehat{D}_u \widehat{D}_v }$
by $H_{\text{inflated}} = \frac{A_{uv}}{\max \{ \widehat{D}_u \widehat{D}_v , 200\} }$, we obtain an error rate of $74/1221$ by thresholding its second eigenvector. This suggests that initial misclassifications are indeed due to low degree nodes (the average degree is $27$, but there are also many leafs present).}

\begin{figure}[!ht]
\begin{minipage}[t]{0.45\linewidth}
\centering
\includegraphics[scale=0.45]{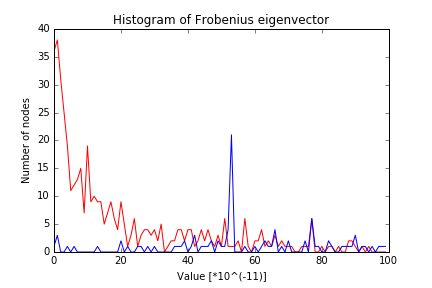}
\caption{Histogram of the Frobenius eigenvector restricted to the roughly $600$ nodes that have corresponding value in the interval $[0,10^{-9}]$. The colors represent the communities.}
\label{fig::concentration}
\end{minipage}
\hspace{0.5cm}
\begin{minipage}[t]{0.45\linewidth}
\centering
\includegraphics[scale=0.45]{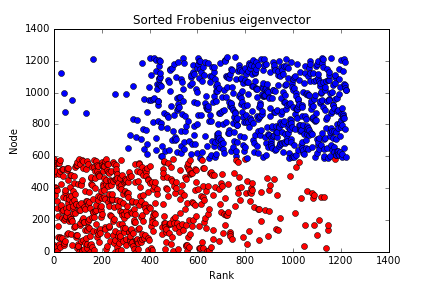}
\caption{Ranking of the $1221$ indices of the Frobenius eigenvector according to an increasing corresponding value. Rank $1$ is the node with smallest value in the eigenvector and rank $1221$ the node with largest value. Colors indicate community-membership.}
\label{fig::ranking}
\end{minipage}
\end{figure} 

\subsection{Interpretation of the conditions}
Note that, since $\mathbb{E}\widehat{D}_u$ is related to $D_u$ according to \eqref{eq::expD}, $\widehat{H}$ normalizes the tendency of communities to connect by the average degree of their nodes and loses therefore some information about the graph. See the observations and remarks below:
\begin{observation}
If, for some $i,j,l \in S$,
\[  \frac{B_{ij}}{M_i} = \frac{B_{lj}}{M_l}, \]
then 
\[ \frac{\E {\# \mbox{edges between community }i \mbox{ and } j }}{ \E { \mbox{total degree of vertices in community } i }} = \frac{ \E { \# \mbox{edges between community }l \mbox{ and } j}}{\E {  \mbox{total degree of vertices in community } l}}. \]
\label{rm::prop}
\end{observation}

\begin{remark}
The identifiability condition is violated if there are distinct $i$ and $l$ and there exists some constant $c > 0$ such that
\[ B_{ij} = c B_{lj} \] for all $j$. Indeed, in that case, $M_i = c M_l$ and thus
\[ \frac{B_{ij}}{M_i} =  \frac{c B_{lj}}{c M_l} = \frac{B_{lj}}{M_l}. \]
However, unlike the setting considered in \cite{LeRi13, Jin15}, it is not necessary for $B$ to be full rank. Indeed, consider 
\[
B = \lr{\begin{matrix} 
1 & 2 & 3 \\
2 & 0 & 2 \\
3 & 2 & 5 \\
\end{matrix}},
\]
which has rank $2$. Let $\alpha_1 = \alpha_2 = \alpha_3 = \frac{1}{3}$ and $ \s{\sigma_u = i}{} D_u = i \alpha_i n   \emph{log}^2(n)$ for all $i = 1,2,3$. Then it is easily verified that the identifiability condition is met. 
\label{rm::violating_ident} 
\end{remark}

Note that $\mathcal{G}(B,K,\{ \sigma_u \}_{u=1}^n,\{ D_u \}_{u=1}^n)$ and $\mathcal{G}(B^*,K,\{ \sigma_u \}_{u=1}^n,\{ D^*_u \}_{u=1}^n)$ generate the same ensemble of random graphs whenever
\[ \frac{D_u B_{\sigma_u \sigma_v} D_v}{ \overline{D}} = \frac{D^*_u B^*_{\sigma_u \sigma_v} D^*_v}{\overline{D^*}}. \]
Hence, the underlying block-matrix $B$ cannot be estimated from a single observation of the graph. Rather,  we may estimate
\be n \overline{D} \approx \sum_u \widehat{D}_u, \ee
and, denoting the assigned community-membership (after applying our reconstruction algorithm) of $l$ by $\tau_l$,
\be \lr{ \sum_u \widehat{D}_u } \frac{\sum_{u:\tau_u = i} \sum_{v:\tau_v = j } \widehat{H}_{uv}}{\lr{\sum_{u:\tau_u = i} 1 } \lr{\sum_{v:\tau_v = j} 1}} \approx  \frac{B_{ij}}{\overline{M}_i \ \overline{M}_j}.  \ee 
Hence, for a DC-SBM $\mathcal{G}(B,K,\{ \sigma_u \}_{u=1}^n,\{ D_u \}_{u=1}^n)$, the matrix
\[ \lr{\frac{B_{ij}}{\overline{M}_i \ \overline{M}_j}}_{i,j=1}^K \]
is identifiable, not $B$. 
It is due to this degeneracy of the DC-SBM and the structure of $\widehat{H}$ that condition \eqref{eq::dist} in Theorem \ref{thm::main} is the best possible:

\begin{lemma}
Consider a DC-SBM $\mathcal{G}(B,K,\{ \sigma_u \}_{u=1}^n,\{ D_u \}_{u=1}^n)$. Fix $i$ and $l$, then the following are equivalent:
\begin{enumerate}
\item for all $j$ we have
\[\frac{B_{ij}}{M_i} = \frac{B_{lj}}{M_l};\]
\item there exist a DC-SBM $\mathcal{G}(B^*,K,\{ \sigma_u \}_{u=1}^n,\{D^*_u\}_{u=1}^n)$, with the same community-structure $\{\sigma_u\}_u$, such that for all $j$,
\[ B^*_{ij} = B^*_{lj} \] and, for all $u,v$,
\[ \frac{D_u B_{\sigma_u \sigma_v} D_v}{ \overline{D}} = \frac{D^*_u B^*_{\sigma_u \sigma_v} D^*_v}{\overline{D^*}}. \]
\end{enumerate}
\label{lm::best_cond}
\end{lemma}

\subsection{Random Walk point of view}
The matrix $\widehat{H}$ is related to a random walk on an instance of the random graph. Indeed,
\[ \widehat{H}_{uv} = \left\{ 
  \begin{array}{l l}
    \frac{1}{\widehat{D_u}} \frac{1}{\widehat{D_v}} A_{uv} = \frac{A_{uv}}{\widehat{D_u}} \frac{A_{vu}}{\widehat{D_v}} & \quad \text{ if } A_{uv} = 1,\\
    0 & \quad  \text{ if } A_{uv} = 0, \\
  \end{array} \right. \]
since $A_{uv} = A_{vu}$ is either $1$ (in case edge $uv$ is present) or $0$ (when $u$ and $v$ are not connected). Now, $\widehat{D_u} = \s{l=1}{n} A_{lu}$, as it is the observed degree, which we denoted here in increasing order: $\widehat{D_1} \leq \widehat{D_2} \leq \cdots \leq \widehat{D_n}$. Thus, $\frac{A_{uv}}{\widehat{D_u}}$ is exactly the probability that a random walk (in an undirected graph without weights) jumps from vertex $u$ to $v$, given that it is currently at vertex $u$. Denoting the latter probability by $\mathbb{P}_u ( u \to v)$, we see that
\[ \widehat{H}_{uv} = \mathbb{P}_u ( u \to v) \mathbb{P}_v ( v \to u) = \mathbb{P}_u ( u \to v \to u), \]
due to the Markov property of the random walk. In other words, $\widehat{H}_{uv}$ is the probability that a random walk currently at vertex $u$ will consecutively traverse edge $uv$ and back. 

Extending this observation to powers of $\widehat{H}$ leads to: \[  (\widehat{H}^k)_{uv} = \s{l_1 = 1 , \ldots, l_{k-1}=1}{n} \mathbb{P}_u ( u \to l_1 \to \ldots \to l_{k-1} \to v)  \mathbb{P}_v ( v \to l_{k-1} \to \ldots \to l_1 \to u), \]
the probability that a random walk, after traversing a path of length $k$ starting at $u$ and ending at $v$, subsequently traverses that path in the exact opposite direction.

Further, note that
\[ (\widehat{D}_1, \ldots, \widehat{D}_n) \widehat{H} = (\mathds{1}_{\widehat{D}_1 \neq 0}, \ldots, \mathds{1}_{\widehat{D}_n \neq 0}), \]
hence if $\textbf{v}$ is an eigenvector of $\widehat{H}$ with eigenvalue $\lambda$, then
\[ \s{u=1}{n} \mathds{1}_{\widehat{D}_u \neq 0} v_u = \lambda  \s{u=1}{n} \widehat{D}_u v_u. \]
Since it can be easily verified that $\widehat{H}$ is primitive on connected components, the Perron-Frobenius theorem implies that the eigenvector $\textbf{v}_{max}$ corresponding to the largest eigenvalue $\lambda_{max}$ (which is positive) has only positive elements. Hence,
\[  \lambda_{max} = \frac{\s{u=1}{n} \mathds{1}_{\widehat{D}_u \neq 0} v_u}{\s{u=1}{n} \widehat{D}_u v_u} \geq \frac{\s{u=1}{n} \mathds{1}_{\widehat{D}_u \neq 0} v_u}{\widehat{D}_n \s{u=1}{n} \mathds{1}_{\widehat{D}_u \neq 0} v_u} = \frac{1}{\widehat{D}_n}.\]

We may derive an upper bound by noting that the spectral radius is bounded from above by the maximal absolute row sum:
\[ \lambda_{max} \leq \max_{u=1}^n \lr{ \s{v=1}{n} \mathbb{P}_u(u \to v \to u) }.  \]

\section{Outline of proof of main theorem}
In this section we consider the setting of Theorem \ref{thm::main}. All lemmas here, except Lemma \ref{thm::super_Log},  assume either \eqref{eq::growth_sup_log} or \eqref{eq::growth_log} to hold.  Lemma \ref{thm::super_Log} assumes condition \eqref{eq::growth_sup_log}: the minimum degree should grow faster than log$(n)$. Lemma \ref{thm::log} assumes \eqref{eq::growth_log}: the minimum degree is of order  log$(n)$.

 Our first objective is to show that $\widehat{H}$ is close to some matrix $P$, in the sense that their difference $W:=\widehat{H} - P$ has negligible spectral radius relatively to that of $P$. Here, an entry $(u,v)$ of $P$ is defined as
\be P_{uv} = \frac{1}{n \overline{D}} \frac{B_{\sigma_u \sigma_v}}{\overline{M}_{\sigma_u} \overline{M}_{\sigma_v}}. \label{eq::Matrix_overlineP} \ee
We relate $P$ in turn to $Z$ defined by
\be Z_{ij} = \frac{\alpha_j B_{ij}}{\overline{M}_i \overline{M}_j}, \quad i,j \in S. \label{eq::Matrix_Z} \ee
Indeed, we show that if $y = (y(1), \ldots, y(K))^T$ is an eigenvector of $Z$ with eigenvalue $\lambda$, then $(y(\sigma_1), \ldots, y(\sigma_n))^T$ fulfils that role for  $P$ with eigenvalue $\frac{ 1	}{\overline{D}}\lambda $. As a consequence, the eigenvectors of $P$ associated to non-zero eigenvalues are constant on blocks.

Finally, we consider the matrix that has the first $L$ eigenvectors of $P$ as its columns. We show that the rows of this matrix cluster to within vanishing distance of block-specific representatives. 
 We start by inspection of the difference
\be W =  \widehat{H} - P = (\widehat{H} - H) + (H - \E{H}) + (\E{H} - P), \label{eq::W}\ee
where
$H$ is defined as
\be H_{uv} = \left\{ 
  \begin{array}{l l}
    \frac{1}{\mathbb{E}\widehat{D}_u \mathbb{E} \widehat{D}_v}A_{uv} & \quad \text{ if } A_{uv} = 1,\\
    0 & \quad  \text{ otherwise }. \\
  \end{array} \right. \ee

Define 
\[ \Delta(P) = \min \{ |\lambda - \mu|: \lambda \neq \mu, \lambda,\mu \text { eigenvalue of } P  \}, \]
i.e., the smallest gap between consecutive eigenvalues. A crucial role will be played by Lemma \ref{lm::alignment_eigenvectors} below, which says that to any eigenvector $\widehat{x}$ of $\widehat{H}$ there exists an eigenvector $x$ of $P$ such that $||x - \widehat{x}|| \to 0$ as $n \to \infty$, whenever 
\[ \frac{\rho(W)}{\Delta(P)} \to 0, \]
as $n \to \infty$, \textcolor{black}{where we recall that $\rho(X)$ denotes the spectral radius of a matrix $X$}.  Hence, we need to calculate  $\Delta(P)$:  
\begin{lemma}
The smallest gap between subsequent eigenvalues of $P$ is given by
\[ \Delta (P) = \Omega \lr{ \frac{1}{ \overline{D}} }.\]
\label{co::GapP}
\end{lemma}

 All terms in the right hand side of \eqref{eq::W} have negligible spectral radius with respect to $\Delta(P)$:
\begin{lemma}
The matrix $\E{H}$ is close to $P$ in the following sense:
\[\rho \lr{\E{H} - P} = \bigO \lr{ \frac{1}{D_1} } \frac{1}{ \overline{D}} = o_n(1) \frac{1}{ \overline{D}}.\]
\label{th::P_OverlineP}
\end{lemma}
\begin{lemma}
The matrix $H$ concentrates with high probability around its expectation, as follows:
\[\rho \lr{H - \E{H} } = \bigO \lr{ \frac{1}{ \sqrt{\emph{log}(n)}}} \frac{1}{ \overline{D} } = o_n(1) \frac{1}{ \overline{D}}.\]
\label{th::H_min_H_Bar}
\end{lemma}

\begin{lemma}
Consider the DC-SBM in the dense regime, where \eqref{eq::growth_sup_log} holds.
Then, for the spectral radius of the difference $\widehat{H} - H$ it holds with high probability that
\[ \rho(\widehat{H} - H) = \bigO \lr{ \sqrt{\frac{\emph{log}(n)}{D_1(n)}} } \frac{1}{\overline{D}(n)} = o_n(1) \frac{1}{ \overline{D}}. \]
\label{thm::super_Log}
\end{lemma}

\begin{lemma}
Consider the DC-SBM in the regime where \eqref{eq::growth_log} holds. 
Then, for the spectral radius of the difference $\widehat{H} - H$ it holds with high probability that
\[ \rho(\widehat{H} - H) = \bigO \lr{ \frac{1}{\emph{log}^{1/3}(n)}  } \frac{1}{\overline{D}(n)} = o_n(1) \frac{1}{ \overline{D}}. \]
\label{thm::log}
\end{lemma}

We use these Lemmas in conjunction with Lemma \ref{lm::alignment_eigenvectors} below to prove:
\begin{lemma}
To each normed eigenvector $\widehat{\textbf{x}}$ of $\widehat{H}$ corresponds a normed eigenvector $\textbf{x}$ of $P$ such that
\[ \widehat{\textbf{x}} \cdot \textbf{x} = 1 - \bigO \lr{ \lr{\frac{\rho(W)}{\Delta(P)}}^2} = 1 - o_n(1), \]
where
\[ \rho(W) \leq \rho(\widehat{H} - H) + \rho(H - \E{H}) + \rho(\E{H} - P). \]
\label{lm::eig_WideH_to_PBar}
\end{lemma}

 Having proved this lemma, we show that Algorithm \ref{alg1}  indeed correctly reconstructs the community of all but a vanishing fraction of vertices. 
 
Recall the definition of $\widehat{H}$ and
observe that $\widehat{H}$ is symmetric. Consequently, there exist $n$ eigenvectors of $\widehat{H}$ that form an \emph{orthonormal} basis: thus, we are indeed able to find $L$ \emph{orthonormal} eigenvectors of $\widehat{H}$ corresponding to its first eigenvectors. 

Next we show that the $(\widehat{z}_u)_{u \in V}$, defined in \eqref{def::lines}, tend to block-representatives:
\begin{lemma}
There exist $K$ vectors $\{t_k\}_{k \in S}$, i.e., block-representatives, such that
\[ || \sqrt{n} \widehat{z}_u - t_{\sigma_u}  || = \bigO \lr{ \lr{\frac{\rho(W)}{\Delta(P)}}^{2/3}} = o_n(1) \]
for all but $\bigO \lr{ n\lr{ \frac{\rho(W)}{\Delta(P)}}^{2/3}} $ nodes.
\label{thm::ev_vanishing_dist}
\end{lemma}

 The remaining and crucial step is to demonstrate that those block-representatives are indeed distinct:
\begin{lemma}
 Assume that for all $i,j$ there exists $i'$ such that
\be \frac{B_{i i'}}{\overline{M}_i \overline{M}_{i'}} \neq \frac{B_{j i'}}{\overline{M}_j \overline{M}_{i'}}, \label{eq::ass_dis} \ee
then
$|t_k - t_l| = \Omega(1)$ for all $k \neq l$.
\label{thm::unique_Representatives}
\end{lemma}

\begin{proof}[Proof of Theorem \ref{thm::main}]
After proving the above lemmas, it remains to show that $\widehat{L}$ in step $(1)$ of Algorithm $1$   with high probability  equals $L$. Further, we should verify that the procedure in step $3$ forms the right clusters.
For the first step notice the following: In the regime where \eqref{eq::growth_sup_log} holds,
\[ \rho(W) = \bigO \lr {\frac{1}{D_1} + \frac{1}{\sqrt{\text{log}(n)}} + \sqrt{\frac{\text{log}(n)}{D_1}} } \frac{1}{\overline{D}}, \]
and in the other regime,  where \eqref{eq::growth_log} holds, 
\[ \rho(W) = \bigO  \lr{ \frac{1}{D_1} + \frac{1}{\sqrt{\text{log}(n)}} + \frac{1}{\text{log}^{1/3}(n)} }\frac{1}{\overline{D}}.  \]
Compare this to $f$ as in Algorithm $1$: depending on the regime, the term in parentheses goes to zero upon division by $f(n)$. To see this, note that due to Bernstein's inequality \eqref{lm::Bernstein}, equation \eqref{eq::bernouilli}, $\widehat{D}_u \in (1/2 \overline{M}_{\sigma_u}, 3/2 \overline{M}_{\sigma_u}) D_u$ for $u = 1$ and $u = n$ with high probability. Hence $\widehat{D}_1$ ( $\widehat{D}_n$ ) is of the same order of magnitude as $D_1$ (respective $D_n$).
Now, due to Lemma \ref{lm::alignment_eigenvectors} below, the first $L$ eigenvalues of $\widehat{H}$ are of order $\frac{1}{\overline{D}} - \bigO(\rho(W)) \gg \frac{f(n)}{\overline{D}}$. The remaining eigenvalues are of order $\bigO(\rho(W)) \ll \frac{f(n)}{\overline{D}}$. Further $D_{\text{average}}$ may be written as twice the sum of $\Omega(n^2)$ independent Bernoulli random variables. It is thus with high probability a constant away from $\overline{D}$. Hence $\widehat{L} = L$ with high probability.

In step $3$, the probability that all picked pairs contain only typical vertices (i.e., whose corresponding rows cluster around $K$ centres) is larger than $(1 - f^{2/3}(n))^{2\tau(n)}$ which tends to one, since $f^{2/3}(n) \tau(n) \to 0$ as $n \to \infty$. Thus, with high probability, for a pair $t$, $\delta(t)$ vanishes in front of $f^{2/3}(n)$ if the vertices in the pair belong to the same community. $\delta(t)$ is of order $\Omega(1)$ otherwise. Hence, $\epsilon$, as defined in step $3$ of Algorithm $1$, is of order $\Omega(1)$, it thus estimates the separation-distance in Lemma \ref{thm::unique_Representatives}.

Further, at most $f(n)^{2/3} \ n$ vertices are not typical. Hence, the chosen ball around $m$ with radius $\epsilon / 8$ contains at least $(f(n)^{1/3} - f(n)^{2/3} ) n \gg f(n)^{2/3} \ n$ typical vertices. Those must necessarily belong to the same community. Since all typical vertices belonging to the same community are at most a distance $f(n)^{2/3}$ apart, all of them are located in the ball of radius $\epsilon/4$ around $m$.

 We see that the algorithm puts, with high probability, all but a vanishing fraction of nodes in $K$ clusters. 
\end{proof}

\section{Algebraic Preliminaries}
We shall make use of the following fact about the spectral radius:
\begin{lemma}
If $|X| \leq Y$ holds entry-wise for two real symmetric matrices $X$ and $Y$, then
\[\rho(X) \leq \rho(Y).\]
\label{lm::rhoBoundCOMP}
\end{lemma}
\begin{proof}
Due to the Rayleigh-Ritz theorem, we have
\[ \rho(X) = \max_{||z||=1} \lVert Xz \rVert.\]
Hence,
\[ \ba 
\rho(X) &= \max_{||z||=1} ||Xz|| \\
&\leq \max_{||z||=1} || \ Y \ |z| \ || \\
&= \max_{||z||=1} || Yz|| \\
&= \rho(Y).
\ea \]
\end{proof}
\noindent
The following lemma could be of independent interest as a simple alternative to the commonly used David-Kahan theorem:
\begin{lemma}
\label{lm::alignment_eigenvectors}
Let $A, \delta A$ be two $n \times n$ symmetric matrices. Let $\lambda_1 \geq \ldots \geq \lambda_n$ be the eigenvalues of $A + \delta A$ and $\mu_1 \geq \ldots \geq \mu_n$ be the eigenvalues of $A$. Let $ \Delta = \min \{ |\mu_i - \mu_j|: \mu_i \neq \mu_j, \mu_i,\mu_j \text { eigenvalue of } A  \}.$ Assume that $\rho \lr{\delta A} < \frac{\Delta}{2}$. Let $\textbf{v}_i$ be a normed eigenvector of $A + \delta A$ corresponding to eigenvalue $\lambda_i$, for any $i = 1,\ldots, n$. Then, 
\begin{enumerate}
\item $|\lambda_i - \mu_i| \leq \rho(\delta A),$
\item the dimension of the eigenspace $E_i$ of $A + \delta A$ corresponding to the eigenvalue $\lambda_i$ is no larger than the dimension of the eigenspace of $A$ corresponding to the eigenvalue $\mu_i$,
\item there exists a normed eigenvector $\widehat{\textbf{v}}_i$ of $A$ corresponding to eigenvalue $\mu_i$ such that 
\[ \textbf{v}_i \cdot \widehat{\textbf{v}}_i \geq \sqrt{1 - \lr{\frac{\rho(\delta A)}{\Delta/2}}^2}. \]
\end{enumerate}
\end{lemma}
\begin{proof}
$(i)$ is due to Weyl's inequality (see for instance \cite{HoJo85}). \\
To prove $(ii)$, let $d$ be the dimension of $E_i$ and write $\lambda_i = \lambda_{i+1} = \cdots = \lambda_{i+d-1}$. Since $|\lambda_i - \mu_i| \leq \rho(\delta A),$ we have $|\mu_i - \mu_{i+1}| \leq 2 \rho(\delta A) < \Delta$. Thus $\mu_i = \mu_{i+1}$, and similarly for the other eigenvalues. 

To prove $(iii)$, we start with some notation: Let $m$ be the number of distinct eigenvalues of $A$, denote those distinct numbers as $\gamma_1 > \cdots > \gamma_m$. Define $S_i = \{ u \in \{1,\ldots,n \} : \mu_u = \gamma_i \},$ the set of indices of eigenvalues that are all equal to $\gamma_i$. For $u \in \{1,\ldots,n \}$, define $\tau_u \in \{1, \ldots,m \}$ as the unique index such that $u \in S_{\tau_u}$. 
 Write 
\[ \textbf{v}_i = \s{j}{} \alpha_j \textbf{w}_j, \]
where $\{\textbf{w}_j\}_j$ are orthonormal eigenvectors of $A$ with associated eigenvalues $\{\mu_j \}_j$. Then,
\[ (A + \delta A)\textbf{v}_i = \s{j}{} \alpha_j \mu_j \textbf{w}_j + (\delta A )\textbf{v}_i.\]
Hence,
\[ (\delta A )\textbf{v}_i = \s{j \notin S_{\tau_i}}{} \alpha_j (\lambda_i - \mu_j) \textbf{w}_j + \s{j \in S_{\tau_i}}{} \alpha_j (\lambda_i - \mu_j) \textbf{w}_i.  \]
Taking norms on both sides,
\[ (\rho(\delta A))^2 \geq \s{j \notin S_{\tau_i}}{} \alpha^2_j (\lambda_i - \mu_j)^2 \geq \s{j \notin S_{\tau_i}}{} \alpha^2_j(\Delta - \Delta/2)^2 = \lr{1- \s{j \in S_{\tau_i}}{} \alpha_j^2}(\Delta/2)^2, \]
because, by definition $|\mu_i - \mu_j| \geq \Delta$ if $\tau_i \neq \tau_j$,  and our observation $|\lambda_i - \mu_i| \leq \rho(\delta A) < \Delta/2.$
Put 
\[ \widehat{\textbf{v}}_i = \frac{1}{\sqrt{\s{j \in S_{\tau_i}}{} \alpha_j^2}}  \s{j \in S_{\tau_i}}{} \alpha_j \textbf{w}_j,  \]
then
\[ \textbf{v}_i \cdot \widehat{\textbf{v}}_i = \sqrt{\s{j \in S_{\tau_i}}{} \alpha_j^2} \geq \sqrt{1 - \lr{\frac{\rho(\delta A)}{\Delta/2}}^2}.  \]
\end{proof}

\begin{lemma}
\label{lm::deviation_random_matrix}
Consider a square $n \times n$ symmetric zero-diagonal random matrix $A$ such that its elements $A_{uv} = A_{vu}$ are independent Bernoulli random variables with parameters 
\[ \E{A_{uv}} = a_{uv} \frac{ \widehat{\omega}(n)}{n}, \]
where the $a_{uv}$ are constants independent of $n$ and $\widehat{\omega}(n) = \Omega( \emph{log}(n) )$. Then, with probability larger than $1 - \bigO \lr{\frac{1}{n^2}}$, the spectral radius of $A - \E{A}$ satisfies
\[ \rho(A - \E{A}) \leq \bigO \lr{ \sqrt{\widehat{\omega}(n)} }. \]
\end{lemma}
\begin{proof}
This is precisely Lemma 2 in \cite{ToMa14}, where we quantified the term \emph{with high probability}. We did this by choosing $c_1 > 3$ in its proof. Note that the latter proof builds further on results by Feige and Ofek \cite{FeOf05}.
\end{proof}

\begin{lemma}[Bernstein's inequality]
Let $X_1, \ldots, X_n$ be zero-mean independent random variables all bounded from above by one. Put $\sigma^2 = \frac{1}{n} \s{u=1}{n} \emph{var}(X_u)$. Then, 
\[ \P{ \frac{1}{n} \s{u=1}{n} X_u > \epsilon } \leq \emph{exp} \lr{ - \frac{n \epsilon^2}{2(\sigma^2 + \epsilon / 3)} }. \]
\label{lm::Bernstein}
\end{lemma} 
\begin{proof}
See \cite{Ber46}.
\end{proof}

 Note that Bernstein's lemma can easily be extended to the case of non-centred random variables. 

\section{Proofs}
In the proofs below, we shall often write 
\be D_u = \phi_u \omega(n), \label{def::phi_u}\ee
 where $1 = \phi_1 \leq \phi_2 \leq \cdots \leq \phi_n$, and 
 \be \omega(n) = D_1 \label{def::omega}. \ee 
Further, we introduce 
\be  g(n) = \s{l = 1}{n} \phi_l, \label{def::g} \ee
\be \overline{\phi}(n) = \frac{g(n)}{n}, \label{def::phi_bar} \ee 

\begin{proof}[Proof of Lemma \ref{co::GapP}]
Write 
\[ P_{uv} = \frac{1}{n \overline{D}} \frac{ B_{\sigma_u \sigma_v}}{\overline{M}_{\sigma_u}  \  \overline{M}_{\sigma_v} } = \frac{1}{n \overline{D}} \frac{Z_{\sigma_u \sigma_v}}{\alpha_{\sigma_v}}. \] 
Let $\textbf{y} = (y(1), \ldots, y(K))^T$ be an eigenvector of $Z$ with eigenvalue $\lambda$, we show that $\textbf{w} = (y(\sigma_1), \ldots, y(\sigma_n))^T$ is an eigenvector of $P$ with eigenvalue $\frac{ 1}{\overline{D}} \lambda .$ Indeed,
\[ \ba
P \textbf{w} &=  \left( \begin{array}{ccc}
\s{l=1}{n} P_{1l} \cdot y(\sigma_l)  \\
\vdots \\
\s{l=1}{n} P_{nl} \cdot y(\sigma_l)  \end{array} \right) \\
&= \frac{ 1}{n \overline{D}} \left( \begin{array}{ccc}
\s{l=1}{n}  Z_{\sigma_1 \sigma_l} / \alpha_{\sigma_l} \cdot y(\sigma_l)  \\
\vdots \\
\s{l=1}{n}  Z_{\sigma_n \sigma_l} / \alpha_{\sigma_l} \cdot y(\sigma_l)  \end{array} \right)\\
&= \frac{ 1}{n \overline{D}} \left( \begin{array}{ccc}
\s{k=1}{K} n \ \alpha_{k} \ Z_{\sigma_1 k} / \alpha_{k} \cdot y(k)  \\
\vdots \\
\s{k=1}{K} n \ \alpha_{k} \  Z_{\sigma_n k} / \alpha_{k} \cdot y(k)  \end{array} \right)\\
&= \frac{ 1}{ \overline{D}} \left( \begin{array}{ccc}
 \lambda y(\sigma_1)  \\
\vdots \\
 \lambda y(\sigma_n)  \end{array} \right) \\
&= \frac{ 1}{\overline{D}} \lambda \textbf{w}.
\ea \] 
Thus $\frac{ 1}{ \overline{D}}\lambda $ is an eigenvalue of $P$. 

 For the other direction, note that if $\sigma_u = \sigma_v$, then row $u$ and row $v$ in $P$ are identical. Hence, if $\textbf{w} = (w(1), \ldots, w(n))^T$ is an eigenvector of $P$ corresponding to a  non-zero eigenvalue, then $w(u) = w(v)$. Let $\textbf{w} = (w(\sigma_1), \ldots, w(\sigma_n))^T$ be an eigenvector of $P$ with eigenvalue $\lambda \neq 0$. By carrying out a similar calculation as above, we see that $(w(1), \ldots, w(K))^T$ is an eigenvector of $Z$ with eigenvalue $\overline{D} \lambda$. 

The statement follows from this one-to-one correspondence between the eigenvectors of both matrices corresponding to non-zero eigenvalues.
\end{proof}

\begin{proof}[Proof of Lemma \ref{th::P_OverlineP}]
Note that
\[ \E{H} - P = \E{H} - (P -\text{diag}(P_{11}, \ldots, P_{nn}) ) +\text{diag}(P_{11}, \ldots, P_{nn}).  \]
Now, 
\[ \rho(\text{diag}(P_{11}, \ldots, P_{nn}) ) = \bigO \lr{ \frac{1}{n \overline{D}}}, \]
as
$\text{diag}(P_{11}, \ldots, P_{nn} )$ contains only $K$ different elements, each of order $\frac{1}{n \overline{D}}$.
Further, for $u \neq v$,
\[ \ba
\E{H_{uv}} &= \frac{D_u}{\E{\widehat{D}_u}} \frac{D_v}{\E{\widehat{D}_v}} B_{\sigma_u \sigma_v} \frac{1}{n \overline{D}} \\
&= (1 + \delta(n)) \frac{ B_{\sigma_u \sigma_v}}{\overline{M}_{\sigma_u}  \  \overline{M}_{\sigma_v} } \frac{1}{n \overline{D}} \\
&= P_{uu} + \delta(n) P_{uu},
 \ea \]
where $\delta(n) =\bigO(\epsilon_n)$, with, due to \eqref{eq::expD},  $\epsilon_n \leq \max_i \frac{2}{\overline{M}_i} \frac{1}{\omega(n)}$  tending to zero uniformly for all $u,v$. 
Hence, due to Lemma \ref{lm::rhoBoundCOMP}, 
\[ \rho \lr{ \E{H} - \lr{ P -\text{diag}(P_{11}, \ldots, P_{nn}) }+ \text{diag}(P_{11}, \ldots, P_{nn}) }= \bigO \lr{ \frac{1}{D_1} } \frac{1}{ \overline{D}}. \]
\end{proof}

\begin{proof}[Proof of Lemma \ref{th::H_min_H_Bar}]
We start by introducing the constants $C_{ij} =  \frac{B_{ij}}{\overline{M}_i \overline{M}_j}$, and $\alpha = \max_{ij} \frac{1}{\overline{M}_i \overline{M}_j}$.
Put for $u < v$,
\[X_{uv} = X_{vu} = \alpha^{-1} \omega^2(n)  \lr{H_{uv} - \E{H_{uv}} }, \]
where $\omega(n)$ is defined in \eqref{def::omega}.
That is,
\[X_{uv} = \frac{1+o(1)}{\alpha} \lr{ \frac{1}{\overline{M}_{\sigma_u} \overline{M}_{\sigma_v}} \frac{1}{\phi_u \phi_v} \text{Ber} \lr{ \frac{D_u D_v}{n \overline{D} }  B_{\sigma_u \sigma_v}  } - C_{\sigma_u \sigma_v} \frac{\omega(n)}{\overline{\phi}} \frac{1}{n} }, \]
with $\phi_u$ and $\overline{\phi}$ defined in \eqref{def::phi_u}, respectively, \eqref{def::phi_bar}.
Due to our choice of $\alpha$ and the assumption that $\phi_u \geq 1$ for all $u$, 
\[ X_{uv} \in (1+o(1)) \left[ - p_{uv}, 1 - p_{uv}  \right], \]
where 
\[p_{uv} = \frac{C_{\sigma_u \sigma_v}}{\alpha} \frac{\omega(n)}{\overline{\phi}} \frac{1}{n}.\]

Let $\widehat{X}_{uv} = \frac{X_{uv}}{1+o(1)}$ such that $\widehat{X}_{uv} \in \left[ - p_{uv}, 1 - p_{uv}  \right] $.
We shall compare the symmetric zero-diagonal matrix $\widehat{X}$ to the deviation  from its expectation of another symmetric zero-diagonal matrix, where elements $uv$ are given by $\text{Ber} \lr{p_{uv}} $, for $u \neq v$. Since by assumption \eqref{eq::degree_hetero},
\be \frac{\omega(n)}{  \overline{\phi}(n)} = \frac{D_1^2(n)}{\overline{D}(n)} =  \Omega( \text{log}(n)), \label{eq::deg_het}\ee
Lemma \ref{lm::deviation_random_matrix} applies. Following an argument given in \cite{ToMa14}, we construct a function $Y_{uv}$ such that  $Y_{uv}$ has values only in $\{ - p_{uv}, 1 - p_{uv} \}$ and $\E{Y_{uv} \left| \widehat{X}_{uv} \right. } = \widehat{X}_{uv}.$ First, let $\{U_{uv} \}_{u<v}$ be independent uniformly distributed random variables. Fix $u < v$. Define, for $x \in \left[ - p_{uv}, 1 - p_{uv}  \right]$ and $w \in [0,1]$,
\[ F_{uv}(x,w) = 1 - p_{uv} - 1_{x \leq w -p_{uv}}, \]
and,
\[ Y_{uv} = Y_{vu} = F_{uv}(\widehat{X}_{uv}, U_{uv}). \]
Then,
\[ \P{F_{uv}(\widehat{X}_{uv}, U_{uv}) = 1 - p_{uv} \left| \widehat{X}_{uv} \right.} = \widehat{X}_{uv} + p_{uv}, \]
and,
\[ \P{F_{uv}(\widehat{X}_{uv}, U_{uv}) =  - p_{uv} \left| \widehat{X}_{uv} \right.} = 1 - p_{uv} - \widehat{X}_{uv} , \]
thus,
\[ \E{Y_{uv} \left| \widehat{X}_{uv} \right. } = \widehat{X}_{uv}, \]
and,
\[ \P{ Y_{uv} = 1 - p_{uv} } = \E{\widehat{X}_{uv}} + p_{uv} = p_{uv}. \]
Hence, indeed, $Y_{uv} = \text{Ber} \lr{p_{uv}} - p_{uv}$.

Let $Y$ be the symmetric zero-diagonal matrix with each element $uv$ given by $Y_{uv}$, for $u \neq v$. Then, according to Lemma \ref{lm::deviation_random_matrix}, 
\be \P{ \rho(Y) \leq \bigO \lr{\sqrt{\frac{\omega(n)}{\overline{\phi}}}}} \geq 1 - \bigO \lr{1 / n^2}. \label{eq::rhoY}\ee
We shall use this observation in the following comparison,
\[ \rho \lr{ \widehat{X} } = \rho \lr{ \E{Y \left| \widehat{X} \right. }} \leq \E{ \rho(Y)  \left| \widehat{X} \right.},  \] 
by Jensen's inequality.
 Put $S = \E{ \rho(Y)  \left| \widehat{X} \right.}$, we shall show that it is also upper bounded by $\bigO \lr{\sqrt{ \omega(n) / \overline{\phi}}}$.

Firstly, note that $|Y|$ is element-wise dominated by the all-one matrix, hence $\rho(Y) \leq n$. Secondly, the sigma-algebra generated by $S$ is contained in the sigma-algebra generated by $\widehat{X}$. Hence,
\[ \E{\rho(Y)|S} = \E{ \left. \E{\rho(Y)\left| \widehat{X} \right.} \right|S} = \E{ S | S} = S. \]

Further, both $Y$ and $\widehat{X}$ take only finitely many different values, and thus  $\rho(Y)$ and $S$ take values in a finite space. It therefore makes sense to consider, for $t > 0$, the function
\[ \beta(\cdot) = \P{\rho(Y) > t | S = \cdot}. \]
We have,
\[ S = \E{\rho(Y)|S} \leq \beta(S)n + (1-\beta(S))t, \]
i.e.,
\[ \beta(S) \geq \frac{S-t}{n-t}. \]
Denote $\gamma = \P{S > t + 1}$, then
\[ \ba \P{\rho(Y) > t} &= \E{\beta(S)} \\
&\geq \E{\beta(S) \mathds{1}_{S > t+1}} \\
&\geq \frac{\gamma}{n-t}.  \ea \]
As a consequence, for $t = \bigO \lr{\sqrt{ \omega(n) / \overline{\phi}}}$, by \eqref{eq::rhoY} one has
\[ \ba  \P{S > t + 1} &= \gamma \leq (n-t)\P{\rho(Y) > t} \\
&= (n-t)\bigO \lr{1 / n^2} \\
&= \bigO \lr{1 / n}. \ea \]
Therefore,
\[ \ba \rho(H - \E{H}) &= \frac{\alpha}{\omega^2(n)} \rho(X) \\
& \leq (1 + o(1))\frac{\alpha}{\omega^2(n)} \rho(\widehat{X}) \\
& \leq \bigO \lr{ \sqrt{\frac{1}{\overline{\phi} \omega^3(n)}} }, \ea \] 
where the first inequality stems from the fact that the order $1 + \epsilon_n$ term in \eqref{eq::expD} holds uniformly over all vertices. 
Finally, due to \eqref{eq::deg_het},
\[ \rho(H - \E{H}) = \bigO \lr{ \sqrt{\frac{\overline{\phi}}{\omega(n)}} \frac{1}{\overline{\phi} \omega(n)} } =   \bigO\lr{ \frac{1}{\sqrt{\text{log}(n)}}  } \frac{1}{ \overline{D}} .  \]

\end{proof}

\begin{proof}[Proof of Lemma \ref{thm::super_Log}]
To prove this theorem we show that in the present setting, with high probability, 
\[ (\widehat{H}-H)_{uv} = \epsilon_{uv} H_{uv},  \]
where, for some constant $\widehat{C}$ and all large enough $n$,
\be |\epsilon_{uv}| \leq \widehat{C} \epsilon(n), \label{eq::eps_uv} \ee 
with
\be  \epsilon(n) :=   \sqrt{\frac{6}{\min_i \overline{M_i}} \frac{2 \text{log}(n)}{\omega(n)} } = \bigO \lr{ \sqrt{\frac{\text{log}(n)}{D_1(n)}} } \to 0, \label{eq::eps} \ee
 by assumption. 
Consequently, after an appeal to Lemma  \ref{lm::rhoBoundCOMP},
\be \rho( \widehat{H} - H) \leq \rho(|\widehat{H} - H|) \leq \widehat{C} \epsilon(n) \rho( H). \label{eq::rho_widH_H}  \ee
Since, $H = \E{H} + H - \E{H}$, it follows from Lemmas \ref{th::P_OverlineP} and \ref{th::H_min_H_Bar} that
\[ \rho(H) = \bigO \lr{\frac{1}{\overline{D}}}, \]
which completes the proof.

 Consider the difference 
\[ \frac{1}{\widehat{D}_u}\frac{1}{\widehat{D}_v} - \frac{1}{\mathbb{E}\widehat{D}_u} \frac{1}{\mathbb{E}\widehat{D}_v} = \frac{1}{\mathbb{E}\widehat{D}_u}\frac{1}{\mathbb{E}\widehat{D}_v} \frac{1}{1 + \frac{\widehat{D}_u - \mathbb{E}\widehat{D}_u}{\mathbb{E} \widehat{D}_u}} \frac{1}{1 + \frac{\widehat{D}_v - \mathbb{E}\widehat{D}_v}{\mathbb{E} \widehat{D}_u}} - \frac{1}{\mathbb{E}\widehat{D}_u}\frac{1}{\mathbb{E}\widehat{D}_v} =  \frac{1}{\mathbb{E}\widehat{D}_u \mathbb{E}\widehat{D}_v} \epsilon_{uv},\]
thus 
\[\epsilon_{uv} = \frac{\mathbb{E}\widehat{D}_u - \widehat{D}_u}{\mathbb{E}\widehat{D}_u} + \frac{\mathbb{E}\widehat{D}_v - \widehat{D}_v}{\mathbb{E}\widehat{D}_v} + \bigO  \lr{ \lr{\frac{\mathbb{E}\widehat{D}_u - \widehat{D}_u}{\mathbb{E}\widehat{D}_u}}^2} + \bigO \lr{ \lr{\frac{\mathbb{E}\widehat{D}_v - \widehat{D}_v}{\mathbb{E}\widehat{D}_v}}^2}.\]
We quantify $\frac{\mathbb{E}\widehat{D}_u - \widehat{D}_u}{\mathbb{E}\widehat{D}_u}$. Since $\widehat{D}_u$ is a sum of Bernoulli random variables with mean
\[ \E{\widehat{D}_u} = D_u \overline{M}_{\sigma_u}(1-o(1)), \]
where the $o(1)$ term follows from \eqref{eq::expD}, we have for $\epsilon(n)$ as in \eqref{eq::eps}, the Bernstein's inequality (see \eqref{lm::Bernstein}),
\be \ba
\P{ \left| \frac{\mathbb{E}\widehat{D}_u - \widehat{D}_u}{\mathbb{E}\widehat{D}_u} \right| > \epsilon(n)} &\leq 2 \text{exp} \lr{-\frac{\epsilon^2(n)}{2 + \epsilon(n)/3} \E{\widehat{D}_u}} \\
&= 2 \text{exp} \lr{-\frac{\epsilon^2(n)}{2 + \epsilon(n)/3}D_u \overline{M}_{\sigma_u}(1-o(1))} \\
&\leq 2 \text{exp} \lr{-\frac{\epsilon^2(n)}{3}  \omega(n) \frac{\overline{M}_{\sigma_u}}{2}} \\
&\leq \frac{2}{n^2}. \label{eq::bernouilli}
 \ea \ee Invoking this we establish the union bound 
\[ \ba
\P{ \left| \frac{\mathbb{E}\widehat{D}_1 - \widehat{D}_1}{\mathbb{E}\widehat{D}_1} \right| \leq \epsilon(n), \ldots, \left| \frac{\mathbb{E}\widehat{D}_n - \widehat{D}_n}{\mathbb{E}\widehat{D}_n} \right| \leq \epsilon(n)} &\geq 1 -  \s{u=1}{n} \P{ \left| \frac{\mathbb{E}\widehat{D}_u - \widehat{D}_u}{\mathbb{E} \widehat{D}_u} \right| > \epsilon(n)} \\
&\geq 1 - \frac{2}{n} \to 1, 
 \ea \]
as $n \to \infty$. Hence, 
\[ E = \left\{ \left| \frac{\mathbb{E}\widehat{D}_u - \widehat{D}_u}{\mathbb{E}\widehat{D}_u} \right| \leq \epsilon(n) \mbox{ for all } u \in V \right\} \]
holds with high probability. Thus we establish \eqref{eq::eps_uv}: $| \epsilon_{uv}| \leq 2 \epsilon(n) + \bigO(\epsilon^2(n)) \leq \widehat{C} \epsilon(n)$, with $\widehat{C}$ a large enough constant. 

 We henceforth condition on $E$. Then, for $u \neq v$,
\[ \ba \widehat{H}_{uv} - H_{uv} &= \lr{\frac{1}{\widehat{D}_u}\frac{1}{\widehat{D}_v} - \frac{1}{\mathbb{E}\widehat{D}_u \mathbb{E}\widehat{D}_v}} A_{uv} \\
&= \epsilon_{uv} \frac{1}{\mathbb{E}\widehat{D}_u \mathbb{E}\widehat{D}_v} A_{uv} \\
&=  \epsilon_{uv} H_{uv}.
 \ea \]
 \end{proof}

\begin{proof}[Proof of Lemma \ref{thm::log}]
We define
 \be \epsilon(n)  = \frac{1}{\text{log}^{1/3}(n)} \label{eq::max_psi}\ee 
and we shall call a vertex $u$ good if $| \mathbb{E}\widehat{D}_u - \widehat{D}_u| \leq \epsilon(n) \mathbb{E}\widehat{D}_u$. We use this definition to split
\be \left| \frac{1}{\widehat{D}_u} \frac{1}{\widehat{D}_v} - \frac{1}{\mathbb{E}\widehat{D}_u} \frac{1}{\mathbb{E}\widehat{D}_v} \right| A_{uv} = M_{uv} + M^{\text{c}}_{uv} + M^{\text{r}}_{uv} - M^{\text{cr}}_{uv}, \label{eq::diff} \ee
where 
\[ M_{uv} = \left| \frac{1}{\widehat{D}_u} \frac{1}{\widehat{D}_v} - \frac{1}{\mathbb{E}\widehat{D}_u} \frac{1}{\mathbb{E}\widehat{D}_v} \right| A_{uv} \mathds{1}_{\{u \text{ and } v \text{ good}\}}, \]
\[M^{\text{c}}_{uv} = \left| \frac{1}{\widehat{D}_u} \frac{1}{\widehat{D}_v} - \frac{1}{\mathbb{E}\widehat{D}_u} \frac{1}{\mathbb{E}\widehat{D}_v} \right| A_{uv} \mathds{1}_{\{v \text{ bad}\}}, \]
\[M^{\text{r}}_{uv} = \left| \frac{1}{\widehat{D}_u} \frac{1}{\widehat{D}_v} - \frac{1}{\mathbb{E}\widehat{D}_u} \frac{1}{\mathbb{E}\widehat{D}_v} \right| A_{uv} \mathds{1}_{\{u \text{ bad}\}}, \]
\[M^{\text{cr}}_{uv} = \left| \frac{1}{\widehat{D}_u} \frac{1}{\widehat{D}_v} - \frac{1}{\mathbb{E}\widehat{D}_u} \frac{1}{\mathbb{E}\widehat{D}_v} \right| A_{uv} \mathds{1}_{\{u \text{ and } v \text{ bad}\}}. \]

We shall show that all terms in \eqref{eq::diff} have a negligible spectral radius compared to $\Delta(P)$. First note that the difference
 \[ \frac{1}{\widehat{D}_u}\frac{1}{\widehat{D}_v} - \frac{1}{\mathbb{E}\widehat{D}_u}\frac{1}{\mathbb{E}\widehat{D}_v},\]
may be written as
\[ \frac{1}{\mathbb{E}\widehat{D}_u \mathbb{E}\widehat{D}_v} \epsilon_{uv}, \]
where \[\epsilon_{uv} = \frac{\mathbb{E}\widehat{D}_u - \widehat{D}_u}{\mathbb{E}\widehat{D}_u} + \frac{\mathbb{E}\widehat{D}_v - \widehat{D}_v}{\mathbb{E}\widehat{D}_v} + \bigO  \lr{ \lr{\frac{\mathbb{E}\widehat{D}_u - \widehat{D}_u}{\mathbb{E}\widehat{D}_u}}^2 }  + \bigO \lr{ \lr{\frac{\mathbb{E}\widehat{D}_v - \widehat{D}_v}{\mathbb{E}\widehat{D}_v}}^2}.\]
Now, similarly as in the proof of Lemma \ref{thm::super_Log}, there exists a constant $\widehat{C}$, such that
$\epsilon_{uv} \leq \widehat{C} \epsilon(n)$ if both $u$ and $v$ are good. Consequently,
$\rho(M) \leq \widehat{C} \epsilon(n) \rho(\widehat{H}).$

 Next we analyse the other terms in \eqref{eq::diff}. We start with $M^{\text{c}}$. The idea is that, although now 
\[ \left| \frac{1}{\widehat{D}_u} \frac{1}{\widehat{D}_v} - \frac{1}{\mathbb{E}\widehat{D}_u} \frac{1}{\mathbb{E}\widehat{D}_v} \right| = \mathcal{O} \lr{ \frac{1}{\mathbb{E}\widehat{D}_u} \frac{1}{\mathbb{E}\widehat{D}_v} }, \]
the total number of non-zero elements in a column of $M^{\text{c}}$ is very small, so that its spectral radius indeed vanishes upon division by  $\Delta(P)$. We note that
\[ \lr{A_{uv} \mathds{1}_{\{v \text{ bad}\}}}_{u,v} = \lr{A_{uv} \mathds{1}_{\{u \text{ bad}\}}}^T_{u,v}, \]
so that a similar statement holds for the maximal row sum of $M^{\text{r}}$. Obviously, $M^{\text{cr}} \leq M^{\text{c}}$, and so do their spectral radii. 

As a consequence of these observations, it thus suffices to prove our claim for $M^{\text{c}}$. To do so, we proceed in three steps:
First we show that 
\be \P{ \mathcal{E}_1 } = \P{ \left \{ \forall u: \left| \frac{\mathbb{E}\widehat{D}_u - \widehat{D}_u}{\mathbb{E}\widehat{D}_u} \right| \leq 1/2\right \} } \geq  1 - 2/n^2. \label{eq::step1} \ee 
From which it follows after a short computation that, with probability larger than $1 - 2/n^2$, for all $u,v$,
\[ \left| \frac{1}{\widehat{D}_u} \frac{1}{\widehat{D}_v} - \frac{1}{\mathbb{E}\widehat{D}_u} \frac{1}{\mathbb{E}\widehat{D}_v} \right| \leq 3 \frac{1}{\mathbb{E}\widehat{D}_u} \frac{1}{\mathbb{E}\widehat{D}_v}. \]
Keeping this in mind, it thus suffices to demonstrate that $(A_{uv} \mathds{1}_{\{v \text{ bad}\}})_{uv}$ has a spectral radius much smaller than the spectral radius of $A$. The column sum in the former equals the number of bad neighbours a vertex has. That is, the spectral radius is bounded by $\max_u X_u$, where for $u \in V$,
\be X_u = \sum_{v \in \mathcal{N}(u)} Z_v, \label{eq::X_u} \ee
with,
\[ Z_v = \mathds{1}_{\{v \text{ is bad}\}}. \] 
Caution is needed here as the indicator functions in \eqref{eq::X_u} are not independent. 

 In the second step we shall show that with high probability the number of edges between vertices in the neighbourhood of $u$ is negligible compared to the expected degree of vertex $u$. That is, 
\be \P{ \mathcal{E}_2(u) } = \P{ \left \{  \sum_{x,y \in \mathcal{N}(u)} A_{xy} \leq \frac{1}{2} \epsilon(n) \omega(n) \right \} } \geq 1 - 2/n^2, \label{eq::E2} \ee
where $\omega(n)$ is defined in \eqref{def::omega}.
 Hence, except for possibly $\frac{1}{4} \epsilon(n) \omega(n)$ of them, the variables in \eqref{eq::X_u} form an independent set (conditional on not having any neighbours among $\mathcal{N}(u)$). 
 
  The last step consists in showing that this leads to 
\be \P{\left. X_u > \epsilon(n) \bigO \lr{\mathbb{E}\widehat{D}_u} \right|  \mathcal{E}_1, \mathcal{E}_2 } = o(1/n). \label{eq::E3} \ee 
The assertion follows now straightforwardly: with high probability, we have 
\[ \ba
\sum_{v} M^{\text{c}}_{uv} &\leq 3 \frac{1}{\mathbb{E}\widehat{D}_u} \max_v \frac{1}{\mathbb{E}\widehat{D}_v} X_u \\&\leq 3 \frac{1}{\mathbb{E}\widehat{D}_u} \max_v \frac{1}{\mathbb{E}\widehat{D}_v} \epsilon(n) \bigO \lr{\mathbb{E}\widehat{D}_u} \\
&\leq \bigO \lr{ \frac{\epsilon(n)}{\min_v \mathbb{E}\widehat{D}_v } } \\
&= \bigO \lr{ \frac{\epsilon(n)}{\omega(n)} }.
 \ea \]
Now $\overline{D} = \bigO(\omega(n))$, since $\frac{D_1^2(n)}{\overline{D}(n)} = \Omega(\text{log}(n))$.
Consequently, due to the choice of $\epsilon(n)$ in \eqref{eq::max_psi},
\[ \rho(M^{\text{c}}) =  \bigO\lr{\frac{\epsilon(n)}{\omega(n)}} = \bigO \lr{ \frac{1}{\text{log}^{1/3}(n)} } \frac{1}{\overline{D}(n)} = o_n(1) \frac{1}{\overline{D}(n)}. \]

The first step, i.e. demonstrating equation \eqref{eq::step1}, is easily carried out:
Fix $u \in V$ and use Bernstein's inequality \eqref{lm::Bernstein} to verify the bound
\[ \P{ \left| \frac{\mathbb{E}\widehat{D}_u - \widehat{D}_u}{\mathbb{E}\widehat{D}_u} \right| > 1/2} \leq 2 \text{exp} \lr{-\frac{3}{26} \E{\widehat{D}_u}}. \]
Now, for $n$ large enough, $\mathbb{E}\widehat{D}_u \geq \overline{M} C_{B,\textbf{M}}  \text{log} n$, and by assumption, $C_{B,\textbf{M}}$ from \eqref{eq::growth_log} is so large that $\frac{3}{26} \E{\widehat{D}_u} > 2 \text{log}(n).$ Hence,
\[ \P{ 1/2 \mathbb{E}\widehat{D}_u \leq \widehat{D}_u \leq 3/2 \mathbb{E}\widehat{D}_u } \geq 1 - 2/n^2. \]

We proceed with the second step, i.e., \eqref{eq::E2}. Put $\overline{M} = \max_{i} \overline{M_i},$ $\overline{B} = \max_{i,j} B_{ij}$. Set 
$C = \max \{1/2 \overline{M},5\overline{M}^2, \overline{B}  \}$. Consider, conditional on $\widehat{D}_u \leq 2 \mathbb{E}\widehat{D}_u$,
\[ \ba  \sum_{x,y \in \mathcal{N}(u)} A_{xy} &= \sum_{x,y\in \mathcal{N}(u)  } \text{Ber} \lr{  B_{\sigma_x \sigma_y} \frac{\phi_x \phi_y \omega(n)}{g(n)} } \\
&\leq \text{Bin} \lr{4 (\mathbb{E}\widehat{D}_u)^2, \overline{B}  \frac{\phi_x \phi_y \omega(n)}{g(n)} } \\
&\leq \text{Bin} \lr{5\overline{M}^2 \phi_u^2 \omega^2(n), \overline{B}  \frac{\phi_x \phi_y \omega(n)}{g(n)} } \\
&\leq \text{Bin} \lr{5\overline{M}^2	 \phi_n^2 \omega^2(n), \overline{B} \frac{ \phi_n^2 \omega(n)}{g(n)} }\\
&\leq \text{Bin} \lr{C	 \phi_n^2 \omega^2(n), C \frac{ \phi_n^2 \omega(n)}{g(n)} },
\ea \]
where $\phi_u$ and $g(n)$ are defined in \eqref{def::phi_u}, respectively \eqref{def::g}.
 We now show that 
\[ \P{\text{Bin} \lr{C \phi_n^2 \omega^2(n), C  \frac{\phi_n^2 \omega(n)}{g(n)} } \geq \frac{1}{2} \epsilon(n) \omega(n)} = o(1/n). \]
First, note that
\be \P{\text{Bin} \lr{C \phi_n^2 \omega^2(n), C  \frac{\phi_n^2 \omega(n)}{g(n)} } \geq  \frac{1}{2} \epsilon(n) \omega(n)} \leq {C \phi_n^2 \omega^2(n) \choose  \frac{1}{2} \epsilon(n) \omega(n)} \lr{ C  \frac{\phi_n^2 \omega(n)}{g(n)} }^{ \frac{1}{2} \epsilon(n) \omega(n)} \label{eq::bin_bound}.  \ee
Using that ${n \choose k} \leq (\frac{ne}{k})^k$, we have
\[ \ba {C \phi_n^2 \omega^2(n) \choose  \frac{1}{2} \epsilon(n) \omega(n)} &\leq \lr{2 Ce \frac{\phi_n^2 \omega(n)  }{\epsilon(n)} }^{ \frac{1}{2} \epsilon(n) \omega(n)} \\
&= \text{exp} \lr{  \frac{1}{2} \epsilon(n) \omega(n) \text{log} \lr{ 2Ce \frac{ \phi_n^2 \omega(n)}{\epsilon(n)}  }  } \\
&\leq \text{exp} \lr{  \frac{c}{2}  \epsilon(n) \omega(n)\text{log} \lr{ g(n) } +  \frac{1}{2} \epsilon(n) \omega(n) \text{log} \lr{ 2Ce} },
 \ea \]
 where $c < \frac{1}{2}$ from \eqref{eq::growth_log} is such that $\frac{\phi_n^2 \omega^2(n) }{\text{log}^{2/3}(n) n^c} \to 0$ (and thus $\frac{\phi_n^2 \omega(n) }{g^c(n)} = \frac{o_n(1)}{\text{log}^{1/3}(n)} \leq \epsilon(n)$, since $g(n) = \Theta(n)$ in the particular setting of this lemma).
Write 
\[ \ba \lr{ C  \frac{\phi_n^2 \omega(n)}{g(n)} }^{\frac{1}{2} \epsilon(n) \omega(n)} &= \text{exp} \lr{ -\frac{1}{2}  \epsilon(n) \omega(n) \text{log} \lr{ (g(n))^{1-c} \frac{(g(n))^{c}}{C \phi_n^2 \omega(n)} } } \\
&\leq \text{exp} \lr{ -\frac{1}{2}  \epsilon(n) \omega(n) \text{log} \lr{ (g(n))^{1-c}  } } \\
&= \text{exp} \lr{ -\frac{1-c}{2}  \epsilon(n) \omega(n) \text{log} \lr{ g(n) } }, \ea\]
if $n$ large enough.
Combining these estimates, we see that \eqref{eq::bin_bound} may be bounded from above by
\[ \ba 
\text{exp} \lr{  -\frac{1 - 2c}{2}  \epsilon(n) \omega(n)\text{log} \lr{ g(n) } + \frac{1}{2} \epsilon(n) \omega(n) \text{log} \lr{ 2Ce } } \\ \leq \lr{  -\frac{1 - 2c}{4}  \epsilon(n) \omega(n)\text{log} \lr{ g(n) } },
\ea \]
 since $g(n) \geq n \gg 2Ce $.
Finally, since $\frac{1-2c }{4} \epsilon(n) \omega(n) \geq 2,$ for large $n$,
\be \P{\mathcal{E}_2} = 1 - \P{\sum_{x,y \in \mathcal{N}(u)  } A_{xy} \geq  \frac{1}{2} \epsilon(n) \omega(n) } \geq 1 -  e^{-\text{log} (g^2(n))} \geq 1 - 1/n^2,\label{eq::interconn} \ee
that is \eqref{eq::E2}.

We proceed with the last step, i.e., establishing \eqref{eq::E3}. Write, 
\[ X_u = \sum_{v \in \mathcal{N}(u): \mathcal{N}(v) \cap \mathcal{N}(u) \neq \emptyset } Z_v  + \sum_{v \in \mathcal{N}(u): \mathcal{N}(v) \cap \mathcal{N}(u) = \emptyset } Z_v. \]
We already know from \eqref{eq::E2} that the  first sum is smaller than $ \frac{1}{2} \epsilon(n) \omega(n)$, with high probability.
The variables in the second sum,  $\{ Z_v \}_{v \in \mathcal{N}(u): \mathcal{N}(v) \cap \mathcal{N}(u) = \emptyset}$, are independent. For such a vertex $v \in \mathcal{N}(u)$ that has no neighbour with $u$ in common, we have $\widehat{D}_v = d'_v + 1$, where
\[ d_v' =  \s{x \notin \mathcal{N}(u), \\ x \neq u}{} \text{Ber} \lr{  B_{\sigma_v \sigma_x} \frac{D_v D_x }{n \overline{D}} }, \]
the degree of $v$ outside $\mathcal{N}(u) \cup \{u\}$. We show that $v$ is a good vertex with high probability, by proving that $d_v'$ concentrates on its mean which on its turn is close to $\E{\widehat{D}_v}$.
Firstly,
define 
\[ \mathbb{E}_*[\cdot] := \E{ \cdot \left| \mathcal{N}(u), \mathcal{E}_2, \widehat{D}_u \leq 2 \E{\widehat{D}_u} \right.},  \]
then
\[ \ba \mathbb{E}_*[d_v']  &= \s{x \notin \mathcal{N}(u),  x \neq u}{}   B_{\sigma_v \sigma_x} \frac{D_v D_x }{n \overline{D}} \\
&= \s{  x \neq v}{}   B_{\sigma_v \sigma_x} \frac{D_v D_x }{n \overline{D}} - \s{x \in  \mathcal{N}(u) \cup \{u\}, x \neq v }{} B_{\sigma_v \sigma_x} \frac{D_v D_x }{n \overline{D}} \\
&\geq \E{\widehat{D}_v} - \bigO \lr{ \frac{\phi_n^2 \omega(n)}{g(n)} }\E{\widehat{D}_u} \\
&= \E{\widehat{D}_v} - \bigO \lr{ \frac{\phi_n^3 \omega^2(n)}{g(n)}   } \\
&= \E{\widehat{D}_v} - o_n(1).
\ea  \]
Secondly, we use Bernstein's inequality \eqref{lm::Bernstein} to prove that $d'_v$ concentrates around $\E{\widehat{D}_v}$ upto a factor $\epsilon(n)$ as in \eqref{eq::max_psi}:
\[ \ba &\P{ d'_v \geq (1+ \epsilon(n)) \mathbb{E}\widehat{D}_v \left| \mathcal{N}(u), \mathcal{E}_2, \widehat{D}_u \leq 2 \E{\widehat{D}_u} \right.} \\
&\leq \text{exp}\lr{- \frac{(\epsilon(n) \mathbb{E}_*d'_v +(1+\epsilon(n))o_n(1))^2}{2\lr{\mathbb{E}_*d'_v +1/3 \lr{\epsilon(n) \mathbb{E}_*d'_v + (1+\epsilon(n))o_n(1)}}}} \\
&\leq \text{exp} \lr{-\frac{(\epsilon(n) \mathbb{E}_*d'_v)^2 \lr{1 + \frac{o_n(1)}{\epsilon(n) \mathbb{E}_*d'_v}} }{ 4 \mathbb{E}_*d'_v}} \\
&\leq  \text{exp} \lr{- C \epsilon^2(n)  \text{log}(n) },
 \ea \]
 where we \emph{redefined} $C = \frac{1}{8}$.
Similarly,
 \[ \P{ d'_v \leq (1 - \epsilon(n)) \mathbb{E}\widehat{D}_v \left| \mathcal{N}(u), \mathcal{E}_2, \widehat{D}_u \leq 2 \E{\widehat{D}_u} \right.} \leq  \text{exp} \lr{-C \epsilon^2(n) \text{log}(n) }. \]
Hence each vertex $v \in \mathcal{N}(u)$ that has no neighbour with $u$ in common is thus a good vertex with probability $2\text{exp} \lr{-C \epsilon^2(n) \text{log}(n) }$. Consequently, conditional on $\mathcal{N}(u), \mathcal{E}_2, \widehat{D}_u \leq 2 \E{\widehat{D}_u}$,
\[  \sum_{v \in \mathcal{N}(u): \mathcal{N}(v) \cap \mathcal{N}(u) = \emptyset } Z_v  \leq \text{Bin} \lr{2 \mathbb{E}\widehat{D}_u , 2 \text{exp} \lr{- C \epsilon^2(n) \text{log}n } }. \]
We have,
\[ \ba &\P{\text{Bin} \lr{2 \mathbb{E}\widehat{D}_u , 2 \text{exp} \lr{- C\epsilon^2(n) \text{log}n } } \geq  \frac{1}{2} \epsilon(n) \mathbb{E}\widehat{D}_u} \\
 &\leq { 2 \mathbb{E}\widehat{D}_u \choose  \frac{1}{2} \epsilon(n) \mathbb{E}\widehat{D}_u} \lr{ 2 \text{exp} \lr{- C \epsilon^2(n) \text{log}n }}^{ \frac{1}{2} \epsilon(n) \mathbb{E}\widehat{D}_u} \\
&\leq \lr{ \frac{4e}{\epsilon(n)}}^{  \frac{1}{2} \epsilon(n) \mathbb{E}\widehat{D}_u} \lr{ 2 \text{exp} \lr{- C \epsilon^2(n) \text{log}n }}^{ \frac{1}{2} \epsilon(n) \mathbb{E}\widehat{D}_u}  \\
&= \text{exp} \lr{ \frac{1}{2} \epsilon(n) \mathbb{E}\widehat{D}_u \lr{\text{log} \frac{8e}{\epsilon(n)} - C\epsilon^2(n) \text{log} n} } \\
&= o(1/n),
 \ea \]
since $\epsilon(n) = 1 / \text{log}^{1/3}(n).$ Hence,
\[ \P{\left. X_u >  \frac{1}{2} \epsilon(n) \lr{\omega(n) +  \mathbb{E}\widehat{D}_u} \right|  \mathcal{E}_1, \mathcal{E}_2 } = o(1/n). \] 

The last step (\eqref{eq::E3}) is completed by noting that $\omega(n) = \bigO \lr{ \mathbb{E}\widehat{D}_u}$.
\end{proof}

\begin{proof}[Proof of Lemma \ref{lm::eig_WideH_to_PBar}]
All matrices in 
\[ W =  (\widehat{H} - H) + (H - \E{H}) + (\E{H} - P), \]
are real and symmetric, hence, combining Lemmas \ref{co::GapP} - \ref{thm::log},
\[ \ba \rho(W) &\leq \rho(\widehat{H} - H) + \rho(H - \E{H}) + \rho(\E{H} - P) \\
&= o_n(1) \frac{1}{ \overline{D}(n)}. \ea \]
Employing Lemma \ref{lm::alignment_eigenvectors} gives that to each eigenvector $\widehat{\textbf{x}}$ of $\widehat{H} = P+W$ corresponds an eigenvector $\textbf{x}$ of $P$ such that
\[ \ba \widehat{\textbf{x}} \cdot \textbf{x} &\geq \sqrt{1 - \lr{\frac{\rho(W)}{\Delta(P)}}^2} \\
&= 1 - \bigO \lr{ \lr{\frac{\rho(W)}{\Delta(P)}}^2} \\
&= 1 - o_n(1),
\ea \]
since $\Delta(P) = \Omega \lr{ 1 / \overline{D}(n) }$.
\end{proof}

\begin{proof}[Proof of Lemma \ref{thm::ev_vanishing_dist}]
Invoking Lemma \ref{lm::eig_WideH_to_PBar}, to each $\widehat{\textbf{x}}_i$ (with eigenvalue $\widehat{\lambda}_i$) there exists a normed eigenvector $\textbf{x}_i$ (with eigenvalue $\lambda_i$) of $P$ such that
\be \widehat{\textbf{x}}_i \cdot \textbf{x}_i = 1 - f_i(n), \label{eq::ev_aligned} \ee
with $f_i(n) = o_n(1)$.
We claim that all $\lambda_i$ are larger than zero (note that we refer here to a set of $\widehat{L}$ eigenvalues).  This can be seen as follows: From Lemma \ref{co::GapP} we know that the first $L$ eigenvalues of $P$ are of order $1/\overline{D}$ and all other eigenvalues are zero. By Lemma \ref{lm::alignment_eigenvectors}, $|\lambda_i - \widehat{\lambda}_i| \leq \rho(W) \ll 1/\overline{D}$, hence the first $L$ eigenvalues of $\widehat{H}$ are also of order $\Omega\lr{1/\overline{D}} - \bigO(\rho(W)) = \Omega\lr{1/\overline{D}}$, and the other $n-L$ are of order $\bigO(\rho(W))$. Now, the $\widehat{L}$ eigenvalues of $\widehat{H}$ that are picked in Step 1 of Algorithm 1 are precisely those whose absolute eigenvalue exceeds $f(n) / \widehat{D}_{\text{average}} = \Omega\lr{f(n)/\overline{D}} \gg \rho(W),$ by construction of $f$ in Section 3. Hence those eigenvalues must necessarily be of order $\Omega\lr{1/\overline{D}}$ (i.e., they are indeed non-zero) and $L = \widehat{L}$ with high probability.  

Since $\textbf{x}_i$ corresponds to a non-zero eigenvalue, it follows from the proof of Lemma \ref{co::GapP} that $\textbf{x}_i$ is constant on each block, i.e., $\textbf{x}_i(u) = \textbf{x}_i(v)$ if $\sigma_u = \sigma_v$. Let $x_i^{(k)}$ be the value of $\textbf{x}_i$ on block $k \in S$. Put 
\be t_k = \sqrt{n} (x_1^{(k)} , \ldots , x_L^{(k)}). \ee
Then,
\[ \ba 
1/n \left| \{ u \in V : || \sqrt{n}\widehat{z}_u - t_{\sigma_u}  ||^2 \geq T^2 \} \right| &\leq \frac{1}{nT^2} \s{m=1}{n} || \sqrt{n}\widehat{z}_u - t_{\sigma_u}  ||^2 \\
&= 1/T^2 \s{u=1}{n} || (\widehat{x}_1^{(u)} , \ldots , \widehat{x}_L^{(u)}) - (x_1^{(\sigma_u)} , \ldots , x_L^{(\sigma_u)}) ||^2 \\
&= 1/T^2 \s{k=1}{L} ||\widehat{x}_k - x_k ||^2 \\
&= 1/T^2 \s{k=1}{L} f_k(n),
\ea \] 
to finish the proof, let $T =\lr{ \s{k=1}{L} f_k(n)}^{1/3} = \bigO \lr{ \lr{\frac{\rho(W)}{\Delta(P)}}^{2/3}} =  o_n(1)$.
\end{proof}

\begin{proof}[Proof of Lemma \ref{thm::unique_Representatives}]
Below we shall make a spectral decomposition in terms of $L$ orthonormal eigenvectors of $Z$ that span the union of all eigenspaces corresponding to non-zero eigenvalues. Recall from the proof of Lemma \ref{co::GapP} how we can obtain the eigenvectors of $Z$ from the eigenvectors of $P$.

Recall that by construction $\{ \widehat{\textbf{x}}_i \}_{i=1}^L$ are orthonormal eigenvectors of $\widehat{H}$ corresponding to non-zero eigenvalues spanning an $L$ dimensional space.  Recall further from the proof of Lemma \ref{thm::ev_vanishing_dist} that the corresponding eigenvectors $\{\textbf{x}_i \}_{i=1}^L$ of $P$ are associated with non-zero eigenvalues. Lemma \ref{lm::alignment_eigenvectors} $(ii)$ entails that the space spanned by those $\{\textbf{x}_i \}_{i=1}^L$ has also dimension $L$. And Lemma \ref{lm::eig_WideH_to_PBar} implies that $\{\textbf{x}_i \}_{i=1}^L$ become an orthonormal set for $n$ tending to infinity (because they become more and more aligned with the orthonormal set $\{ \widehat{\textbf{x}}_i \}_{i=1}^L$).

Let, as in the proof of Lemma \ref{thm::ev_vanishing_dist}, $x_i^{(k)}$ be the value of $\textbf{x}_i$ on block $k \in S$. 
Note that $\sum_k n \alpha_k (x_i^{(k)})^2 = 1$ for $i \in \{1,\ldots, L\}$. Putting $y_i = \sqrt{n} (x_i^{(1)}, \ldots, x_i^{(K)})^T$, we see that each $y_i$ is a \emph{normalized} eigenvector of $Z$ in the sense that  $\sum_k  \alpha_k (y_i(k))^2 = 1$.
\\ 
Now, assume for a contradiction that $|t_k - t_l| \to 0$ as $n \to \infty$: 
\be \s{i=1}{L} |\sqrt{n} x_i^{(k)} - \sqrt{n} x_i^{(l)}|^2 = \s{i=1}{L} |y_i(l) - y_i(k)|^2 \to 0. \label{eq::ev_equal_lines} \ee
We conclude that there exist orthonormal eigenvectors of $Z$, $\{ \overline{y}_1,\ldots, \overline{y}_L  \}$ (with eigenvalues $\{\lambda_i\}_{i=1}^L$ after a possible relabelling of indices), that span the range of $Z$, such that
\[ \overline{y}_u(k) = \overline{y}_u(l)  \]
for all $u$. The other $K-L$ eigenvectors have zero as an eigenvalue. 

To proceed, consider matrix 
\[ N = \lr{\sqrt{\alpha_u}  \frac{B_{uv}}{\overline{M}_u \overline{M}_v} \sqrt{\alpha_v}}_{u,v}. \]

If $(x(1),\ldots,(x(K))^T$ is an eigenvector of $Z$ then $(\sqrt{\alpha_1} x(1),\ldots,\sqrt{\alpha_K}x(K))^T$ is an eigenvector of $N$, as is easily verified. Hence $N$ has $\{ ( \sqrt{\alpha_1} \overline{y}_i(1), \ldots, \sqrt{\alpha_K} \overline{y}_i(K))^T  \}_{i=1}^{L}$ as eigenvectors corresponding to non-zero eigenvalues and $K-L$ eigenvectors with $0$ as eigenvalue (which do not contribute to the spectral decomposition of $N$). Hence
\[ N = \lr{ \s{l=1}{L}  \sqrt{\alpha_u}  \overline{y}_l(u) \lambda_l  \sqrt{\alpha_v} \overline{y}_l(v) }_{u,v}. \]
Thus, for all $u$,
\[ \ba
 \frac{B_{ku}}{\overline{M}_k \overline{M}_u} &= \sum_m \overline{y}_m(k) \lambda_m \overline{y}_m(u) \\
&= \sum_m \overline{y}_m(l) \lambda_m \overline{y}_m(u) \\
&=  \frac{B_{lu}}{\overline{M}_l \overline{M}_u}, 
 \ea \]
violating assumption \ref{eq::ass_dis}. 
\end{proof}

\subsection{Comparison to spectral analysis on the adjacency matrix}

\begin{proof}[Proof of Theorem \ref{thm::power_law}]
This proof leans strongly on ideas borrowed from \cite{MiPa02}, where graphs \emph{without} a community-structure are considered. Parts of their proof carry through for the DC-SBM considered here.
Note that $\lim_{n \to \infty} g(n)/n = 1$. 

By definition, we require without lose of generality $D_1 \leq D_2 \leq \cdots \leq D_n$. However, we  obtain the same graph (with now a decreasing degree-sequence) by a rearrangement of indices, if we put
\be \phi_u = \left\{ 
  \begin{array}{l l}
    \frac{\phi_1}{u} & \quad \text{ if } u \leq 1 \leq k= n^{\beta}\\
    1 & \quad  \text{ if } u > n^{\beta}, \\
  \end{array} \right. \ee
  where $\phi_1 = n^{\gamma + \beta}$, and $D_u = \phi_u \omega(n)$ (with $\omega$ as in \eqref{def::omega}).
  \be \sigma_u = \left\{ 
  \begin{array}{l l}
    1 & \quad \text{ if } u \leq \frac{n}{2}\\
    2 & \quad  \text{ if } u > \frac{n}{2}. \\
  \end{array} \right. \ee
Denote a sample of the random graph by $G$. We decompose $G$ into the following graphs (exactly as in \cite{MiPa02}) :
\begin{itemize}
\item $G_1$, which is a union of vertex disjoint stars $S_1, \ldots, S_k$. Star $S_u$ has as its center node $u$  and as leaves those vertices from among $\{k+1, \ldots, n\}$ adjacent to $u$, but not adjacent to $\{1, \ldots, u-1\}$;
\item $G_1'$ is the graph consisting of all edges of $G$ with one endpoint in $\{1, \ldots, k\}$ and the other   endpoint in $\{k+1, \ldots, n\}$, except for those edges in $G_1$;
\item $G_2$ is the subgraph of $G$, which is induced by $\{1, \ldots, k\}$;
\item $G_3$ is the subgraph of $G$, which is induced by $\{k+1, \ldots, n \}$.
\end{itemize}
Further, let $F_u$ be the subset of vertices in $\{k+1, \ldots, n \}$ that are adjacent to $\{1, \ldots, u-1\}$  and let $C$ be a  constant, independent of $n$, whose value might change along the course of the proof.

 We claim that $\widehat{d}_u$, the degree of vertex $u$ in $G_1$, concentrates around its mean. Indeed, consider
\[ \widehat{d}_u = \s{l=k+1}{n} \text{Ber} \lr{  \frac{D_u D_l}{g(n) \omega(n)} B_{\sigma_u \sigma_l} } - \s{l \in F_u}{} \text{Ber} \lr{  \frac{D_u D_l}{g(n) \omega(n)} B_{\sigma_u \sigma_l} }, \] 
where $g$ is defined in \eqref{def::g}.
Then, 
\[ d_u = \E{\widehat{d}_u} \geq \frac{\omega(n) \phi_u}{g(n)} \lr{\s{l=k+1}{n} B_{\sigma_u \sigma_l} - C \E{|F_u|}},  \]
which we bound from below by estimating $\E{|F_u|}$, for $u \leq k = n^{\beta}$: For large enough $n$,
\[ \ba \E{|F_u|} &= \s{l=k+1}{n} \s{v=1}{u-1} \frac{D_l D_v}{\omega(n)g(n)} B_{\sigma_u \sigma_l} \\
&\leq C \frac{\omega(n) \phi_1}{g(n)} \s{l=k+1}{n} \s{v=1}{u-1} \frac{1}{v} \\
&\leq C \omega(n) \phi_1 \frac{n - n^{\beta}}{g(n)} n^{\beta} \\
&\leq C \omega(n) n^{\gamma + 2 \beta}, \ea \]
after recalling the special choice for the degree sequence. 
 
Consequently, we have
\[ \frac{n}{g(n)} \frac{B_{11} + B_{12}}{2} D_u \geq d_u \geq \frac{B_{11} + B_{12}}{2} D_u \lr{\frac{n}{g(n)} - C \frac{\omega(n) n^{\gamma + 2 \beta}}{g(n)} }.  \]
Invoking large deviation theory on $\widehat{d}_u$ (which is a sum of Bernoulli random variables), we deduce that
\be \P{|\widehat{d}_u - d_u| > \sqrt{c' d_u \text{log} n}} \leq 2 / n^{c'/4},  \label{eq::concentration_stars} \ee
For $c'>0$ a constant.
We take $c' = 8$ to establish \eqref{eq::concentration_stars} uniformly over all vertices. We next investigate $\Delta(G_1)$, the smallest gap between different eigenvalues of $G_1$. This graph is the union of vertex disjoint stars with degree $\widehat{d}_u$ so that its spectrum is given by
\[ \{ \pm \sqrt{\widehat{d}_1 - 1} , \ldots, \pm \sqrt{\widehat{d}_k - 1} \}. \]
We claim that 
\be \Delta(G_1) \geq C \sqrt{\omega(n)} n^{\frac{\gamma - 3 \beta}{2}} \to \infty \label{eq::gap_G1} \ee
with high probability. 
Indeed, define
\[ x_u^{\pm} = d_u \pm \sqrt{c' d_u \text{log} n}, \]
and note that with high probability  $\widehat{d}_u \geq x_u^{-}$ and $\widehat{d}_{u+1} \leq x_{u+1}^{+}$. To investigate the difference $x_u^{-} - x_{u+1}^{+}$, we first bound $d_u - d_{u+1}$ from below:
\[ \ba
d_u - d_{u+1} &\geq \frac{B_{11} + B_{12}}{2} \omega(n) \phi_1 \lr{ \frac{n/g(n)}{u(u+1)} - C \frac{\omega(n) n^{\gamma + 2 \beta}}{n} \frac{1}{u}} \\
&= \frac{B_{11} + B_{12}}{2} \omega(n) \phi_1 \frac{1}{u} \lr{\frac{n/g(n)}{u+1} - C \frac{\omega(n) n^{\gamma + 2 \beta}}{n} } \\
&\geq \frac{B_{11} + B_{12}}{2} \omega(n) \phi_1 \frac{1}{u} \lr{\frac{n/g(n)}{n^{\beta}+1} - C \frac{\omega(n) n^{\gamma + 2 \beta}}{n} } \\
&\geq \frac{B_{11} + B_{12}}{4} \omega(n) \phi_1 \frac{1}{n^{\beta}} \frac{1}{n^{\beta}} \\
&= \frac{B_{11} + B_{12}}{4}  \frac{\omega(n) n^{\gamma + \beta}}{n^{2 \beta}} \\
&= \frac{B_{11} + B_{12}}{4} \omega(n) n^{\gamma - \beta}.
 \ea \]
Next we show that the $\sqrt{d_u \text{log} n}$ terms are negligible:
\[ \ba 
\sqrt{d_u \text{log} n} &\leq \sqrt{\frac{B_{11} + B_{12}}{2} n/g(n) D_u \text{log} n } \\
&\leq C \sqrt{\omega(n)\text{log}(n) n^{\gamma + \beta}} \\
&\leq C \omega(n) n^{\frac{\gamma + \beta}{2}} \\
&\ll \omega(n) n^{\gamma - \beta},
\ea \]
due to \eqref{eq::C_beta_gamma_2}. Hence,
\[ x_u^{-} - x_{u+1}^{+} \geq C \omega(n) n^{\gamma - \beta}. \] 
As a consequence, 
\[ \ba 
\Delta(G_1) &\geq \min_{u \in \{1, \ldots, k\}} \lr{ \sqrt{\widehat{d}_u - 1} - \sqrt{\widehat{d}_{u+1} - 1} } \\ 
&\geq \min_{u \in \{1, \ldots, k\}} \lr{ \sqrt{x_{u}^{-} - 1} - \sqrt{x_{u+1}^{+} - 1} }  \\
&= \min_{u \in \{1, \ldots, k\}} \lr{ \frac{x_{u}^{-} - x_{u+1}^{+}}{\sqrt{x_{u}^{-} - 1} + \sqrt{x_{u+1}^{+} - 1}}  } \\
&\geq  C \frac{\omega(n) n^{\gamma - \beta}}{\sqrt{\omega(n)} n^{\frac{\gamma + \beta}{2}}} \\
&= C \sqrt{\omega(n)} n^{\frac{\gamma - 3 \beta}{2}},
\ea \]
that is \eqref{eq::gap_G1}.

We continue with an inspection of $G_1'$, that is, we focus on $\widehat{m}_u = \widehat{D}_u | G_1'$, the degree of vertex $u$ in $G_1'$, and show that 
\be \widehat{m}_u \leq  2 c' \text{log} n \label{eq::degree_bound_G_1'}, \ee
with high probability (here, $m_u$ is the expectation of $\widehat{m}_u$). We shall use this in combination with the fact that the spectral radius of a graph is bounded by its largest degree. \\ Write
\[ \widehat{m}_u = \s{l \in F_u}{} \text{Ber} \lr{  \frac{\phi_u \omega(n)}{g(n) } B_{1 \sigma_l} }. \]
This expression allows us to deduce an upperbound for $m_u$,
\[ \ba 
m_u &= \E{\widehat{m}_u} \\
&\leq  C \E{F_u} \frac{\phi_u \omega(n)}{g(n) } \\
&\leq C \omega(n) n^{\gamma + 2 \beta} n^{\gamma + \beta} \frac{1}{u} \frac{\omega(n)}{g(n)}  \\
&\leq C \omega^2(n) \frac{n^{2\gamma + 3\beta}}{n},
\ea \]
which tends to zero due to \eqref{eq::C_beta_gamma_1}. Standard bounds for Bernoulli random variables give
\[ \ba
\P{ |m_u - \widehat{m}_u| \leq c' \text{log} n} &\leq 2 \ \text{exp} \lr{- \frac{(c' \text{log} n)^2}{2(m_u + c' \text{log}(n)/3)}} \\
&\leq 2 \ \text{exp} \lr{ - \frac{1}{4} c' \text{log} n} \\
&= \frac{2}{n^{c'/4}}.
 \ea \]
We conclude that, with probability at least $1 - \frac{2}{n^2}$,
\[ \widehat{m}_u \leq m_u + c' \text{log} n \leq 2 c' \text{log} n, \] i.e.,
\eqref{eq::degree_bound_G_1'} holds.
An identical estimate holds when $u > k$.

We next bound the number of edges in $G_2$, denoted by $E(G_2)$. The square root of $E(G_2)$ is an upper bound for the spectral radius of $G_2$. 
\[ \ba \E{|E(G_2)|} &= C \s{u=1}{k} \s{v=1}{k} \frac{\phi_u \phi_v \omega(n)}{g(n)} \\
&\leq C \frac{n^{\gamma + \beta} n^{\gamma + \beta} \omega(n)  n^{\beta} n^{\beta}}{g(n)} \\
&\leq C \frac{n^{2 \gamma + 4 \beta}}{n},
\ea \]
vanishing for large $n$. Again, upon invoking standard large deviation theory, we have, with probability at least $1 - \frac{2}{n^2}$,
\be \E{|E(G_2)|} \leq 2 c' \text{log} n. \label{eq::edge_bound_G_2} \ee

Consider the degree of a vertex $u > k$ in $G_3$,
\[ \ba \E{ \widehat{D}_u |_{G_3}} &= \s{v=k}{n} \frac{\phi_u \phi_v \omega(n)}{g(n)} B_{\sigma_v \sigma_u} \\ 
&\leq C \frac{\omega(n)}{g(n)} n \\
&\leq C \omega(n).
\ea \]
Hence,
\be \P{\widehat{D}_u |_{G_3} > C \omega(n) + \sqrt{c' \text{log}(n)  C \omega(n) } } \leq \frac{2}{n^{c'/ 4}}. \label{eq::degree_bound_G_3} \ee

Combining these observations leads to our assertion that the first $k$ eigenvectors of $A$  become undistinguishable of those of the $k$ stars, when $n$ tends to infinity. Indeed, split $A$ according to the described  graph-composition:
\[ A = A|_{G_1} +  A|_{G'_1} +  A|_{G_2} + A|_{G_3}, \]
and note that the spectral radii of $A|_{G'_1}$, $A|_{G_2}$ and $A|_{G_3}$ vanish in the presence of $\Delta(G_1)$. This follows because (as mentioned above) for any graph its spectral radius is bounded by the minimum of its largest degree and the square root of its number of edges. Hence, due to \eqref{eq::degree_bound_G_1'} - \eqref{eq::degree_bound_G_3}, 
\[ \rho(A|_{G'_1}) \leq 2 c' \text{log} n,  \] 
\[ \rho(A|_{G_2}) \leq \sqrt{2 c' \text{log} n},  \] 
and
\[ \rho(A|_{G_3}) \leq C \omega(n) + \sqrt{ c' \text{log}(n) C \omega(n)},  \]
with high probability.  
All those three bounds vanish indeed upon division by $\Delta(G_1) \geq C \omega(n) n^{\frac{\gamma - 3 \beta}{2}}$.
 Lemma \ref{lm::alignment_eigenvectors} finishes the proof.
\end{proof}

\subsection{Interpretation of the conditions}

\begin{proof}[Proof of Remark \ref{rm::prop}]
Assume, 
\[  \frac{B_{ij}}{M_i} = \frac{B_{lj}}{M_l} \]
then,
\be  \frac{B_{ij}}{M_i M_j} = \frac{B_{lj}}{M_l M_j}. \label{eq::CondIdentifiabilityEq} \ee
Now, put $\overline{\phi}_i = \frac{1}{\alpha_i n} \s{\sigma_u = i}{}\phi_u$, then
\[ \frac{B_{ij}}{M_i M_j} = \frac{\alpha_i \overline{\phi}_i  B_{ij} \alpha_j \overline{\phi}_j}{\alpha_i \overline{\phi}_i M_i M_j \alpha_j \overline{\phi}_j}. \]
We  give a probabilistic interpretation to the terms appearing in the denominator:
\be \ba
n \alpha_i \overline{\phi}_i M_i &= n \alpha_i \overline{\phi}_i \s{k=1}{K}\sum_{u: \sigma_u = k} \phi_u B_{i \sigma_u}  \\
&= n \alpha_i \overline{\phi}_i \s{k=1}{K} \  n \alpha_k \overline{\phi}_k B_{i k} \\
&= \s{k=1}{K} \ (n \alpha_i) (n \alpha_k) \overline{\phi}_i B_{i k} \overline{\phi}_k \\
&= \s{k=1}{K} \sum_{u: \sigma_u = i} \phi_u \sum_{v: \sigma_v = k} \phi_v B_{ik} \\ 
&= \frac{n}{\omega(n)} \s{k=1}{K} \sum_{u: \sigma_u = i}  \sum_{v: \sigma_v = k}  \P{l \leftrightarrow m}  \\ 
&= \frac{n}{\omega(n)}  \sum_{u: \sigma_u = i}  \s{m=1}{n}  \P{l \leftrightarrow m}  \\ 
&= \frac{n}{\omega(n)} \{ \mbox{expected total degree of vertices in community } i \}.
\ea \ee
An inspection of the numerator reveals
\be \ba
n \alpha_i \overline{\phi}_i B_{ij} \overline{\phi}_j \alpha_j n   &= \sum_{u: \sigma_u = i} \phi_u \sum_{v: \sigma_v = j} \phi_v B_{ij} \\
&= \frac{n}{\omega(n)} \sum_{u: \sigma_u = i}  \sum_{v: \sigma_v = j}  \P{u \leftrightarrow v} \\
&= \frac{n}{\omega(n)}  \{\mbox{expected } \# \mbox{edges between community }i \mbox{ and } j \}
\ea \ee
\end{proof}

\begin{proof}[Proof of Lemma \ref{lm::best_cond}]
Assume first that for some $i$ and $l$ we have for all $j$ 
\[ \widehat{B}_{ij} = \widehat{B}_{lj} \] and, for all $u$, $v$, 
\[ \frac{\phi_u B_{\sigma_u \sigma_v} \phi_v}{g(n)} = \frac{\widehat{\phi}_u \widehat{B}_{\sigma_u \sigma_v} \widehat{\phi}_v}{\widehat{g}(n)}, \]
with $\phi_u$ defined in \eqref{def::phi_u} and $g$ in \eqref{def::g} ($\widehat{\phi}_u$ and $\widehat{g}$ are defined analogously).
Fix $j$. Let $\alpha,\beta$ and $\gamma$ be \emph{any}	indices such that $\sigma_{\alpha} =i, \sigma_{\beta} =j$ and $\sigma_{\gamma} =l$. Then,
\[ \frac{\widehat{\phi}_{\alpha} \widehat{B}_{i j} \widehat{\phi}_{\beta}}{\widehat{g}(n)} = \frac{\phi_{\alpha} B_{i j} \phi_{\beta}}{g(n)} \Rightarrow B_{ij} = \frac{\widehat{\phi}_{\alpha}}{\phi_{\alpha}} \frac{\widehat{\phi}_{\beta}}{\phi_{\beta}} \frac{g(n)}{\widehat{g}(n)} \widehat{B}_{i j}  \]
and
\[ \frac{\widehat{\phi}_{\gamma} \widehat{B}_{l j} \widehat{\phi}_{\beta}}{\widehat{g}(n)} = \frac{\phi_{\gamma} B_{l j} \phi_{\beta}}{g(n)} \Rightarrow B_{lj} = \frac{\widehat{\phi}_{\gamma}}{\phi_{\gamma}} \frac{\widehat{\phi}_{\beta}}{\phi_{\beta}} \frac{g(n)}{\widehat{g}(n)} \widehat{B}_{l j}, \]
implying that (since $\widehat{B}_{ij} = \widehat{B}_{lj}$) 
\[ B_{ij} = \frac{\widehat{\phi}_{\alpha}}{\phi_{\alpha}} \frac{\widehat{\phi}_{\beta}}{\phi_{\beta}} \frac{g(n)}{\widehat{g}(n)}  \widehat{B}_{l j} = \frac{\widehat{\phi}_{\alpha}}{\phi_{\alpha}} \frac{\phi_{\gamma}}{\widehat{\phi}_{\gamma}} B_{l j}. \]
Since $j$ was arbitrary, there exist $c$ such that for all $j$
\[ B_{ij} = c B_{lj}, \]
hereby violating the identifiability condition, as pointed out in Remark \ref{rm::violating_ident}, i.e.,
\[\frac{B_{ij}}{M_i} = \frac{B_{lj}}{M_l},\]
for all $j$. 

 Now assume that $(a)$ holds, that is
\[\frac{B_{ij}}{M_i} = \frac{B_{lj}}{M_l},\]
for all $j$. Define for $k,l \in S$ and $u \in V$
\[ \widehat{B}_{kl} = \frac{1}{M_k} \frac{1}{M_l} B_{kl} \]
and 
\[ \widehat{\phi}_u = f(n) \phi_u M_{\sigma_u}, \]
where
\[ f(n) = \frac{ \sum_v \phi_v M_{\sigma_v}}{\sum_w \phi_w}. \]
Then, 
\[ \widehat{B}_{ij} = \frac{1}{M_i} \frac{1}{M_j} B_{ij}  = \frac{1}{M_i} \frac{1}{M_j} \frac{M_i}{M_l} B_{lj} = \frac{1}{M_j} \frac{1}{M_l} B_{lj} = \widehat{B}_{lj}, \]
and (as above, we define $\widehat{g}$ analogously to $g$),
\[ \ba \frac{\widehat{\phi}_u \widehat{B}_{\sigma_u \sigma_v} \widehat{\phi}_v}{\widehat{g}(n)} &= \frac{1}{\sum_w f(n) \phi_w M_{\sigma_w}} \phi_u M_{\sigma_u} f(n) \frac{B_{\sigma_u \sigma_v}}{M_{\sigma_u} M_{\sigma_v}} f(n) M_{\sigma_v} \phi_v \\
&= \frac{\phi_u B_{\sigma_u \sigma_v} \phi_v}{\sum_w \phi_w} \\
&= \frac{\phi_u B_{\sigma_u \sigma_v} \phi_v}{g(n)}. \ea \]  
\end{proof}

\section{Future research}
\subsection{Exact recovery}
\textcolor{black}{The obtained clustering here is almost-exact: only a vanishing fraction of nodes is miss-classified. It is plausible that an exact clustering could be obtained from this clustering, by using it as input to the "clean-up" algorithm presented in Section $7.2$ of \cite{AbBaHa14} or alternatively, Algorithm $2$ in \cite{MoNeSlSTOC15}.}

\subsection{Non-constant $B$}
\textcolor{black}{In the underlying paper we assumed $B$ to be a constant matrix. The current analysis could be extended to a setting where $B$ is allowed to change with $n$. We need however the existence of a constant $\delta > 0$ such that for all $n$,  $\rho(Z) \geq \delta $ for $\widehat{H}$ to concentrate. For identifiability we need the existence of some $\epsilon > 0$ such that for all $i,j$ and $n$,  $\max_{i'} \left| \frac{B_{i i'}}{\overline{M}_i \overline{M}_{i'}} \neq \frac{B_{j i'}}{\overline{M}_j \overline{M}_{i'}} \right| \geq \epsilon$.}

\subsection{Sparser graphs}
The main issue with both the normalized adjacency matrix and the Laplacian is proving when those matrices concentrate around a deterministic matrix. For the Laplacian, if the degrees are of order $\Omega(\text{log}(n))$, matrices concentrate according to \cite{ChRa11}. But, if the minimum degree is of order $o(\text{log}(n))$, the graph is seen to have some isolated vertices. Those contribute to multiple zeros in the spectrum: hence the matrix does not concentrate. There are multiple ways to overcome this issue, for instance removing the low-degree vertices or raising all the degrees. The latter strategy is proposed in \cite{LeVe15} for the inhomogeneous \ErdosRenyi \  random graph (where edges are independently present with probabilities $(p_{uv})_{u,v=1}^n$) and also in \cite{QiRo13,ChChTs12} (see Section \ref{ss::regularized}) for the the DC-SBM.  According to  \cite{LeVe15}, for $\tau \sim d$, with $d = n \max_{uv} p_{uv}$, with high probability, 
\[ \rho \lr{L_{\tau}  - \lr{  \E{D_{\tau}}^{-1/2} \E{A} \E{D_{\tau}}^{-1/2}  } } = \bigO \lr{\frac{1}{\sqrt{d}}}, \]
where $L_{\tau}$ is defined in \eqref{eq::reg_laplacian}.

Based on these observations, it might be fruitful to use $\widehat{H}$ on a graph where the degrees have been artificially inflated.  

\newpage

\bibliographystyle{plain}
\bibliography{literature}

\end{document}